\documentclass[12pt]{amsart}

\usepackage{snowshovelingalg}
\usepackage{graphicx, subcaption, mdframed, float}
\usepackage[norelsize]{algorithm2e}
\makeatletter
\@namedef{subjclassname@2020}{%
  \textup{2020} Mathematics Subject Classification}
\makeatother

\begin{document}

\author{Mohit Bansil}
\address{Department of Mathematics, Michigan State University}
\email{bansilmo@msu.edu}

\author{Jun Kitagawa}
\address{Department of Mathematics, Michigan State University}
\email{kitagawa@math.msu.edu}

\title[Newton for OT with storage fees]{A Newton algorithm for semi-discrete optimal transport with storage fees}
\thanks{JK's research was supported in part by National Science Foundation grant DMS-1700094.}

\begin{abstract}
We introduce and prove convergence of a damped Newton algorithm to approximate solutions of the semi-discrete optimal transport problem with storage fees, corresponding to a problem with hard capacity constraints. This is a variant of the optimal transport problem arising in queue penalization problems, and has applications to data clustering. Our result is novel as it is the first numerical method with proven convergence for this variant problem; additionally the algorithm applies to the classical semi-discrete optimal transport problem but does not require any connectedness assumptions on the support of the source measure, in contrast with existing results. Furthermore we find some stability results of the associated Laguerre cells. All of our results come with quantitative rates. We also present some numerical examples.
\end{abstract}

\maketitle

%\tableofcontents

\section{Introduction}

\subsection{Semi-discrete optimal transport with storage fees}

In this paper we deal with the following problem. Let $X\subset \R^n$, $n\geq 2$ be compact and $Y:=\{y_i\}_{i=1}^N \subset \R^n$ a fixed collection of finite points, along with a \emph{cost function} $c: X\times Y\to \R$ and a \emph{storage fee function} $F: \R^N\to\R$. We also fix a Borel probability measure $\mu$ with $\spt \mu\subset X$, and assume $\mu$ is absolutely continuous with respect to Lebesgue measure. The \emph{semi-discrete optimal transport with storage fees} is then find a pair $(T, \weightvect)$ with $\weightvect=(\weightvect^1,\ldots, \weightvect^N)\in \R^N$ and $T: X\to Y$ measureable satisfying
$
	T_\#\mu = \sum_{i=1}^N \weightvect^i \delta_{y_i}
$, 
such that
\begin{align}\label{eqn: monge ver}
	\int_X c(x, T(x)) d\mu + F(\weightvect) = \min_{\tilde \weightvect\in \R^N,\ \tilde{T}_\#\mu = \sum_{i=1}^N \tilde\weightvect^i \delta_{y_i}} \int_X c(x, \tilde{T}(x)) d\mu + F(\tilde\weightvect).
\end{align}
In \cite{BansilKitagawa19a} the authors have shown under appropriate conditions, existence of solutions to the problem with storage fees, along with a dual problem with strong duality, and a characterization of dual maximizers and primal minimizers. It is not difficult to see that an optimal $T$ can be constructed via an $\mu$-a.e. partition of the domain $X$ which is induced by a maximizing dual potential, the cells forming such a partition are known as \emph{Laguerre cells} (see Definition \ref{def: lag cell}). 

To contrast, the classical (semi-discrete) optimal transport problem would be to fix a discrete probability measure $\nu$ supported on $Y$, and to find a measurable mapping $T: X\to Y$ such that $T_\#\mu(E):=\mu(T^{-1}(E))=\nu(E)$ for any measurable $E\subset Y$, and $T$ satisfies
\begin{align}\label{eqn: monge}
	\int_X c(x, T(x)) d\mu = \min_{\ti T_\#\mu=\nu} \int_X c(x, \ti T(x)) d\mu,
\end{align}
and it is easy to see the classical problem is a special case of the problem with storage fee above (see paragraph below).

In this paper, we propose and show convergence of a damped Newton algorithm, when the storage fee function is of the form
\begin{align}
	F(\weightvect)=F_{w}(\weightvect):=
	\begin{cases}
		0, &\weightvect\in \prod_{i=1}^N[0, w^i],\\
		+\infty,&\text{else},
	\end{cases}
\end{align}
where $w=(w^1,\ldots, w^N)\in \R^N$ is some fixed vector with nonnegative components. The minimization \eqref{eqn: monge ver} with this choice of $F$ corresponds to a problem where the $i$th target point has a hard capacity constraint given by $w^i$, with no other associated cost of storage. It is clear that if $w$ satisfies $\sum_{i=1}^N w^i=1$, the solution of the problem with storage fee solves the classical optimal transport problem with target measure $\nu=\sum_{i=1}^N w^i\delta_{y_i}$, hence this variant includes the classical case. 

\subsection{Contributions of the paper}
The major novelties of our algorithm above are mainly the following three aspects. First, this is the first algorithm available for problems with storage fee. Second, our method applies to classical optimal transport where the source measure does not satisfy a Poincar{\'e}-Wirtinger inequality, which is a crucial condition in existing results such as \cite{KitagawaMerigotThibert19}. Third, we give explicit errors on  the geometric structures arising in the approximations generated by our algorithm.

We introduce some preliminary notions in Section \ref{section: setup} below to state our damped Newton algorithm, as such we defer the precise statements of our main theorems to Section \ref{section: setup}, along with the outline for the remainder of the paper.  In Theorem \ref{thm: linear convergence}, we show the above mentioned damped Newton algorithm has global linear convergence, and local superlinear convergence. This result is a significant improvement over \cite{KitagawaMerigotThibert19} by the second author, in that the algorithm applies to the wider class of problems with storage fees, but also because the convergence of the algorithm is shown without a connectedness assumption on the support of the source measure (see Remark \ref{rmk: PW} below). It should also be noted that the convergence proof is not a straightforward application of the analysis in \cite{KitagawaMerigotThibert19}. In Theorems \ref{thm: symmetric convergence} and \ref{thm: hausdorff convergence}, we utilize the results of \cite{BansilKitagawa20a} to show explicit convergence rates for the Laguerre cells. We show convergence in terms of both the $\mu$ symmetric difference and Hausdorff distance.

\subsection{Literature analysis}
The optimal transport with storage fees first appears in \cite{CrippaJimenezPratelli09} in the context of queue penalization. The problem also corresponds to the ``lower level problem'' in the bilevel location problem, see \cite{MallozziPassarelli17}. This interpretation is also related to a problem of monopolistic pricing analyzed in \cite{CarlierMallozzi18}. These are a few of the potential applications of the optimal transport problem with storage fees (see also Remark \ref{rmk: data} below), and we emphasize that this paper provides the first numerical method for this problem.

For the classical optimal transport problem (see \cite{PeyreCuturi19} for an excellent overview) there are now many numerical methods. We briefly mention three popular approaches: entropic regularization, discretization schemes for solving the Monge-Amp{\`e}re equation with second boundary value condition, or approximation of a semi-discrete problem.

Entropic regularization is accomplished by adding the \emph{relative (Shannon) entropy} with respect to the tensor product of the source and target measures to the objective functional in the measure valued (Kantorovich) problem, to act as a regularizing term. Numerically, the problem can be solved using the Sinkhorn algorithm (first done for optimal transport by Cuturi, \cite{Cuturi13}). This is generally a fast and parallelizable method: when transporting between two discrete measures supported on $N$ points each, the Sinkhorn algorithm finds approximations with total transport cost within $\epsilon$ of the true value in $O(\frac{N^2\log N}{\epsilon^3})$ operations (see \cite{AltschulerWeedRigollet17}). However, the entropic regularization method has the disadvantage that solutions of the regularized problem are only known to converge in a weak sense to the true solution (weak convergence of measures, see \cite{CarlierDuvalPeyreSchmitzer17}), with no explicit convergence rates.

 For absolutely continuous source and target measures, the solution of the optimal transport problem can be constructed by solving a PDE of Monge-Amp{\`e}re type with the \emph{second boundary value} condition. Finite difference schemes for the Monge-Amp{\`e}re operator with these boundary conditions have been investigated by Benamou, Froese, and Oberman; in \cite{Oberman08, FroeseOberman11a, FroeseOberman11b, BenamouFroeseOberman14} they show various schemes are monotone, stable, and consistent, hence approximations converge uniformly to a viscosity solution of the PDE, (via Barles and Souganidis: \cite{BarlesSouganidis91}). This approach applies to problems with absolutely continuous measures, and some of the schemes mentioned are robust for singular solutions of the PDE.  However, no explicit convergence rates are available for these methods in the optimal transport case (\cite{NochettoZhang2019} gives quantitative rates for the Dirichlet problem assuming higher regularity of solutions). Also, stencils need to be modified near the boundary for these schemes, which is difficult for complicated geometries. Lastly, convexity of solutions is essential, hence these schemes are limited to the ``classical'' Monge-Amp{\`e}re case, $c(x, y)=\norm{x-y}^2$.

The method we use is based on the duality theory for semi-discrete transport problems. By Brenier's theorem (\cite{Brenier91}), solutions of the semi-discrete optimal transport problem can be constructed from a finite envelope of a certain family of functions depending on the cost (when $c(x, y)=\norm{x-y}^2$, the family is affine functions). For the classical Monge-Amp{\`e}re equation this construction goes back to Aleksandrov and Pogorelov, \cite{Aleksandrov05}. The papers \cite{CaffarelliKochenginOliker99, Kitagawa14, AbedinGutierrez17} propose a non-Newton type iterative method, the last result is applicable to \emph{generated Jacobian equations}; a class more general than optimal transport. These results give an upper bound on the number of iterations necessary, but are slower with a bound of $O(\frac{N^4}{\epsilon})$ steps for an error $\epsilon>0$ with target measure supported on $N$ points.

The first use of a Newton method with this  envelope construction appears to be \cite{OlikerPrussner88} for a semi-discrete Monge-Amp{\`e}re equation with Dirichlet boundary condition; there local convergence is proved,  global convergence was later shown in \cite{Mirebeau15}; their setting is for weak solutions of \emph{Aleksandrov} type, generally different from optimal transport solutions.
For the classical optimal transport problem, \cite{AurenhammerHoffmannAronov98} observed that finding the optimal map is equivalent to  extremizing the so-called Kantorovich functional;  \cite{Merigot11} observed good empirical behavior of Newton type methods for this problem (but without convergence proofs).  A damped Newton method is used for the quadratic cost on the torus in \cite{LoeperRapetti05, SaumierAguehKhouider15} with proofs of convergence based on regularity theory of the Monge-Amp{\`e}re equation due to Caffarelli (\cite{Caffarelli92}). In \cite{KitagawaMerigotThibert19}, a damped Newton algorithm is proposed that applies to a wider class of cost functions and global linear and local superlinear convergence for H\"older continuous source measures is proved. A key assumption is that the source measure satisfy a Poincar{\'e}-Wirtinger inequality, a quantitative connectivity assumption on the support (see also \cite{MerigotMeyronThibert18}). Advantages of the semi-discrete approach is that it produces exact solutions to some transport problem, and some methods can be applied to a wide variety of cost functions other than the quadratic distance cost.

\begin{rmk}[Data clustering]\label{rmk: data}
One application of the problem we consider here is to data clustering. Suppose there is some data set that is so large, it can be viewed as being distributed according to an absolutely continuous measure $\mu$. The goal is then to partition the data into $N$ clusters, where for each cluster a ``representative element'' $y_i$ is given. It is natural to utilize semi-discrete optimal transport in data clustering where the affinity of the data considered is measured by the cost function $c$, unlike when using classical optimal transport, by adding a storage fee function one does not have to \emph{a priori} specify the sizes of each cluster. In particular, a storage fee of the form $F_w$ will yield such a clustering, with the hard constraint that the $i$th cluster can be no larger than $w^i$. In this context, each Laguerre cell is a cluster, and it is useful to analyze convergence of these cells in any approximation. 
\end{rmk}

\subsection{Strategy of proof and obstacles}

There are a number of difficulties that prevent a direct translation of the damped Newton algorithm from \cite{KitagawaMerigotThibert19} to the problem with storage fees. First, in the classical case one fixes a discrete target measure $\nu=\sum_{i=1}^N \weightvect^i\delta_{y_i}$, and the Newton algorithm is used to approximate the weight vector $\weightvect=(\weightvect^1,\ldots, \weightvect^N)$. However, in our problem with storage fees, the weight vector $\weightvect$ itself must be chosen as part of the minimization and hence is not fixed, thus it is not even \emph{a priori} clear what quantity to approximate with a Newton algorithm. Additionally, unlike the classical problem, it is possible that $\weightvect^i=0$ for one or more of the entries in an optimal choice for the weight vector, but the algorithm from \cite{KitagawaMerigotThibert19} uses the assumption that all $\weightvect^i$ have strictly positive lower bounds in a crucial way to obtain the convergence. To remedy these issues, we will first approximate the storage function $F_w$: we will use the characterization for solutions found in \cite{BansilKitagawa19a} to find approximating storage functions $\ti F_w$, along with minimizers of the problem \eqref{eqn: monge ver} with $F=\ti F_w$. However, a second difficulty arises as the functions of the form $F_w$ have both highly singular behavior in their subdifferentials at the boundary of their effective domains, while being nonstrictly convex everywhere. Thus, we will further replace functions of this form with uniformly convex, smooth approximations. This procedure turns out to have a regularizing effect on the problem, which allows us to obtain convergence without the aforementioned connectedness assumption as in \cite{KitagawaMerigotThibert19} (see also Remark \ref{rmk: first order cond}).

Concerning the proof of the convergence of the Laguerre cells, we first prove a lemma on the strong convexity of a functional associated to the semi-discrete optimal transport problem. This lemma is then used to control the difference of the Laguerre cells of the problems associated to $\ti F_w$ and those of our uniformly convex, smooth approximations. From here we are able to apply the results of \cite{BansilKitagawa20a} to obtain the desired convergence.

\section{Setup}\label{section: setup}
\subsection{Notation and conventions}
Here we gather notation and conventions to be used in the remainder of the paper. As mentioned above, we fix positive integers $N$ and $n$ and a collection $Y:=\{y_i\}_{i=1}^N\subset \R^n$. The standard $N$-simplex will be denoted
\begin{align*}
	\weightvectset:=\{{\weightvect}\in \R^N\mid \sum_{i=1}^N\weightvect^i=1,\ \weightvect^i\geq 0\},
\end{align*}
and to any vector ${\weightvect}\in \weightvectset$ we associate the discrete measure $\displaystyle\nu_{{\weightvect}}:=\sum_{i=1}^N\weightvect^i\delta_{y_i}$. The notation $\onevect$ will refer to the vector in $\R^N$ whose components are all $1$. We also reserve the notation $\norm{V}:=\sqrt{\sum_{i=1}^N \abs{V^i}^2}$ for the Euclidean ($\ell^2$) norm of a vector $V\in \R^N$, while $\norm{V}_1:=\sum_{i=1}^N \abs{V^i}$ and $\norm{V}_\infty:=\max_{i\in\{1,\ldots, N\}}\abs{V^i}$ will respectively stand for the $\ell^1$ and $\ell^\infty$ norms. We also write $\norm{M}$ for the operator norm of a matrix $M$, the distinction from the Euclidean norm of a vector should be clear from context.

Given any set $A$, we write $\ind(x\mid A):=
\begin{cases}
0,&x\in A,\\
+\infty,&x\not\in A,
\end{cases}$ for the \emph{indicator function} of the set $A$, and for any vector $w\in\R^N$ with nonnegative entries, we denote $F_w:=\sum_{i=1}^N\ind(\cdot\mid [0, w^i])=\ind(\cdot\mid \prod_{i=1}^N[0, w^i])$. We will also use $\L$ to denote the $n$-dimensional Lebesgue measure and $\mathcal{H}^k$ for the $k$-dimensional Hausdorff measure.

Regarding the cost function $c$, we will generally assume the following standard conditions from optimal transport theory:
\begin{align}
	c(\cdot, y_i)&\in C^2(X), \forall i\in \{1, \ldots, N\},\label{Reg}\tag{Reg}\\
	\nabla_xc(x, y_i)&\neq \nabla_xc(x, y_k),\ \forall x\in X,\ i\neq k.\label{Twist}\tag{Twist}
\end{align}
We also assume the following condition, originally studied by Loeper in \cite{Loeper09}.
\begin{defin}
	We say $c$ satisfies \emph{Loeper's condition} if for each $i\in \{1, \ldots, N\}$ there exists a convex set $X_i\subset \R^n$ and a $C^2$ diffeomorphism $\cExp{i}{\cdot}: X_i\to X$ such that 
	\begin{align*}
		\forall\ t\in\R,\ 1\leq k, i\leq N,\ \{p\in X_i\mid -c(\cExp{i}{p}, y_k)+c(\cExp{i}{p}, y_i)\leq t\}\text{ is convex}.\label{QC}\tag{QC}
	\end{align*}
	See Remark \ref{rmk: brenier solutions} below for a discussion of these conditions.
	
	We also say that a set $\tilde{X}\subset X$ is \emph{$c$-convex} with respect to $Y$ if $\invcExp{i}{\tilde{X}}$ is a convex set for every $i\in \{1, \ldots, N\}$. 
\end{defin}
It will be convenient to also introduce $c$-convex functions and the $c$ and $c^*$-transforms. In the semi-discrete case the $c^*$-transform of a function defined on $X$ will be a vector in $\R^N$, while the $c$-transform of a vector in $\R^N$ will be a function whose domain is $X$.
\begin{defin}\label{def: c-transforms}
	If $\varphi: X\to \R\cup \{+\infty\}$  ($\varphi\not\equiv+\infty$) and $\psi\in \R^N$, their $c$- and $c^*$-transforms are a vector $\varphi^c\in \R^N$ and a function $\psi^{c^*}: X\to \R\cup \{+\infty\}$ respectively, defined by
	\begin{align*}
		(\varphi^c)^i:=\sup_{x\in X}(-c(x, y_i)-\varphi(x)),\quad (\psi^{c^*})(x):=\max_{i\in \{1, \ldots, N\}}(-c(x, y_i)-\psi^i).
	\end{align*}
	If $\varphi: X\to \R\cup \{+\infty\}$ is the $c^*$-transform of some vector in $\R^N$, we say \emph{$\varphi$ is a $c$-convex function}. %
	A pair $(\varphi, \psi)$ with $\varphi: X\to \R\cup \{+\infty\}$ and $\psi\in \R^N$ is a \emph{$c$-conjugate pair} if $\varphi=\psi^{c^*}$ and $\psi=\psi^{c^*c}$.
\end{defin}
\begin{defin}\label{def: lag cell}
	For any $\psi\in \R^N$ and $i\in \{1, \ldots, N\}$, we define the \emph{$i$th Laguerre cell} associated to $\psi$ as the set
	\begin{align*}
		\Lag_i(\psi):=\{x\in X\mid -c(x, y_i)-\psi^i=\psi^{c^*}(x)\}.
	\end{align*}
	We also define the function $G: \R^N\to \weightvectset$ and the set $\mathcal{K}^\epsilon$ for any $\epsilon\geq 0$ by,
	\begin{align*}
		G(\psi):&=(G^1(\psi), \ldots, G^N(\psi))=(\mu(\Lag_1(\psi)), \ldots, \mu(\Lag_N(\psi))),\\ \mathcal{K}^\epsilon:&=\{\psi\in \R^N\mid G^i(\psi)> \epsilon,\ \forall i\in \{1, \ldots, N\}\}.
	\end{align*}
\end{defin}
\begin{rmk}\label{rmk: brenier solutions}
	The above conditions \eqref{Reg}, \eqref{Twist}, \eqref{QC} are the same ones assumed in \cite{KitagawaMerigotThibert19}. As also mentioned there, \eqref{Reg} and \eqref{Twist} are standard in the existence theory of optimal transport, while \eqref{QC} holds if $Y$ is a finite set sampled from from a continuous space, and $c$ is a $C^4$ cost function satisfying what is known as the \emph{Ma-Trudinger-Wang} condition (along with an additional convexity assumption on the domain of $c$, which we do not detail here). The strong Ma-Trudinger-Wang condition was first introduced in \cite{MaTrudingerWang05}, and in \cite{TrudingerWang09} in a weaker form. The condition is also \emph{necessary} for the regularity theory of the Monge-Amp{\`e}re type equation arising in optimal transport, see \cite{Loeper09}.
	
	If $\mu$ is absolutely continuous with respect to Lebesgue measure, under \eqref{Twist} the Laguerre cells associated to different indices are disjoint up to sets of $\mu$-measure zero. Then by the generalized Brenier's theorem \cite[Theorem 10.28]{Villani09}, for any vector $\psi\in \R^N$ it is known that the $\mu$-a.e. single valued map $T_\psi: X\to Y$ defined by $T_\psi(x)=y_i$ whenever $x\in \Lag_i(\psi)$, is a minimizer in the classical optimal transport problem \eqref{eqn: monge}, where the source measure is $\mu$ and the target measure is defined by $\nu=\nu_{G(\psi)}$.
\end{rmk}
In order to introduce the damped Newton algorithm we will analyze for our problem \eqref{eqn: monge ver}, we must introduce a few more pieces of notation. The motivation for these will be explained in detail in the following section.

\begin{defin}
	For $h> 0$ and $\thresh\geq 0$ define $g: \R\to \R$ and $w_{h, \thresh}: \R^N\to \R^N$ by
	\begin{align*}
		g(t) :&= 2\( 1+t^2 - t\sqrt{1+t^2} \),\quad
		w^i_{h, \thresh}(\psi):=(G^i(\psi)-\thresh)g\(\frac{\psi^i}{h}\).
	\end{align*}
	Also, we write for any $\epsilon_0>0$ and $w\in \R^N$ with nonnegative entries,
\begin{align*}
 \mathcal{W}^{\epsilon_0}:=\{\psi\in \R^N\mid w^i_{h, \thresh}(\psi)\geq \epsilon_0,\ \forall i\in \{1, \ldots, N\}\},\quad \normset{w, h, \thresh}:=\{\psi\in \mathcal{K}^\thresh\mid \sum_{i=1}^N w^i=\sum_{i=1}^N w_{h, \thresh}^i(\psi)\}.
\end{align*}
\end{defin}
We now use the above notation to propose the following damped Newton algorithm to approximate solutions of \eqref{eqn: monge ver}. Note below, we do not lose any generality in assuming $w^i\leq 1$ for each $i$, as $\mu$ is a probability measure.

\RestyleAlgo{boxruled}
\begin{algorithm}[H]
	\begin{description}
		\item[Parameters] Fix $h$, $\thresh>0$, and $w \in \R^N$ such that $\sum_{i=1}^N w^i \geq 1$, $w^i\in [0, 1]$.
		\item[Input] A tolerance $\zeta > 0$ and an initial $\psi_0\in
		\R^N$ such that
		\begin{equation}\label{eq:nonzerocells}
		\eps_0 :=  \frac{1}{2} 
		\min\left[\min_{i} w_{h, \thresh}^i(\psi_0),~ \min_{i} w^i\right] >  0.
		\end{equation}
		\item[While] $\norm{w_{h, \thresh}(\psi_k) - w}  \geq \zeta$
		\begin{description}
			\item[Step 1] Compute $\vec{d}_k = - [D w_{h, \thresh}(\psi_k)]^{-1} (w_{h, \thresh}(\psi_k) - w)$
			\item[Step 2] For each $\ell\in \N$ let $r_\ell\in \R$ be such that $\psi_{k+1, \ell} :=
			\psi_k + 2^{-\ell} \vec{d}_k+r_\ell\onevect$ satisfies $\psi_{k+1, \ell}\in \normset{w, h, \thresh}$.
			
			\item [Step 3] Determine the minimum $\ell \in \N$ such that $\psi_{k+1, \ell}$ satisfies
			\begin{equation*}
				\left\{
				\begin{aligned}
					&\min_i w_{h, \thresh}^i(\psi_{k+1, \ell}) \geq \eps_0 \\
					&\norm{w_{h, \thresh}(\psi_{k+1, \ell}) - w} \leq (1-2^{-(\ell+1)}) \norm{w_{h, \thresh}(\psi_k) - w}
				\end{aligned}
				\right.
			\end{equation*}
			\item [Step 4] Set $\psi_{k+1} = \psi_k + 2^{-\ell}  \vec{d}_k+r_\ell\onevect$ and $k\gets k+1$.
		\end{description}
	\end{description}
	\label{alg: damped newton}
	\caption{Damped Newton's algorithm}
\end{algorithm}
We now give some heuristics on our algorithm. For $h$, $\thresh\geq 0$ fixed, define for any $t_0\geq 0$, the function $ \approxfee_{t_0, h}: \R\to\R$ by
\begin{align}
	\approxfee_{t_0, h}(t) &
	=\begin{cases}
		-h\sqrt{t(t_0-t)} & \text{if } t \in [0,t_0] \\
		+\infty & \text{else}
	\end{cases},
\end{align}
and for any $w\in \R^N$, $w^i\geq 0$, the function $F_{w, h, \thresh}: \R^N\to \R\cup\{+\infty\}$ by
\begin{align}
	F_{w, h, \thresh}(\weightvect) &= \sum_{i=1}^N \approxfee_{w^i, h}(\weightvect^i-\thresh)+\ind(\weightvect\mid \weightvectset)\\
	&=
	\begin{cases}
		\displaystyle-h \sum_{i=1}^N \sqrt{(\weightvect^i- \thresh)(w^i-\weightvect^i + \thresh )}, &\weightvect\in \weightvectset\cap \prod_{i=1}^N[\thresh, w^i+\thresh],\notag\\
		+\infty,&\text{else}.
	\end{cases}
\end{align}
It can be seen that $F_{w, h, \thresh}$ is a uniformly convex approximation to $F_{w}=F_{w, 0, 0}$ when $h$, $\thresh>0$. Detailed calculations will be deferred to Proposition \ref{prop: constructing solutions} in the following section, but if $\psi\in \R^N$ is a vector such that $w_{h, \thresh}(\psi)=w$, using the results of \cite{BansilKitagawa19a} it can be seen for the map $T_{\psi}$ defined as in Remark \ref{rmk: brenier solutions}, the pair $(T_{\psi}, G(\psi))$ is the unique solution to the minimization problem \eqref{eqn: monge ver} with storage fee function given by $F_{w, h, \thresh}$. Thus the algorithm generates a vector $\psi$ and a storage fee function $\tilde F$ approximating the original $F_w$, such that $(T_\psi, G(\psi))$ solves the optimal transport problem with storage fee $\tilde F$. The normalization $\psi\in \normset{w, h, \thresh}$ at each step in Algorithm \ref{alg: damped newton} is necessary in order to ensure that the magnitude of the error vector $w_{h, \thresh}(\psi_k)-w$ will actually go to zero.

The main theorem of our paper is the following on convergence of the above algorithm. Also, see Definition \ref{def: universal} below for the notion of a universal constant.
\begin{thm}\label{thm: linear convergence}
	Suppose $c$ satisfies \eqref{Reg}, \eqref{Twist}, and \eqref{QC}. Also suppose $X$ is a bounded set that is $c$-convex with respect to $Y$, $\mu=\rho dx$ for some density $\rho\in C^{0, \alpha}(X)$ for some $\alpha\in (0, 1]$, and $\spt\mu\subset X$. Then if $h\in (0,  1]$, $\thresh\in (0, \frac{1}{2N})$, and $\sum_{i=1}^N w^i\geq 1$, Algorithm \ref{alg: damped newton} converges globally with linear rate, and locally with superlinear rate $1+\alpha^2$.
	
	Specifically, the iterates of Algorithm \ref{alg: damped newton} satisfy
	\begin{align*}
	\norm{w_{h, \thresh}(\psi_{k+1}) - w} \leq (1 - \conj \tau_k/2 )\norm{w_{h, \thresh}(\psi_{k}) - w}
	\end{align*}
	where
	\begin{align*}
	\conj \tau_k := \min\(\frac{\eps_0^{\frac{1}{\alpha^2}} \kappa^{1+ \frac{1}{\alpha^2}}}{(8L\ti L^{1+\alpha} \sqrt{N})^{\frac{1}{\alpha^2}}  \norm{w_{h, \thresh}(\psi_k)-w} N^{\frac{1}{\alpha^2}}  }, 1\),
	\end{align*}
	where $L$ and $\kappa$ are as in Proposition \ref{prop: Dw estimates}, and $\ti L \leq \frac{C}{h^{18}\thresh^{9}}$ for some universal constant $C$. 
	
	In addition as soon as $\conj \tau_k = 1$ we have
	\begin{align*}
	\norm{w_{h, \thresh}(\psi_{k+1}) - w} \leq \frac{2L \ti L^{1+\alpha}\sqrt{N} \norm{w_{h, \thresh}(\psi_k)-w}^{1+\alpha^2}}{\kappa^{1+\alpha^2}}.
	\end{align*}

\end{thm}

\begin{rmk}\label{rmk: first order cond}
In \cite{KitagawaMerigotThibert19}, the goal is to find a root of the mapping $G-\beta$ which is in fact the gradient of the concave dual functional in the Kantorovich problem. In contrast, our mapping $w_{h, \thresh}-w$ is not the gradient of any scalar function (seen easily as $Dw_{h, \thresh}$ is not symmetric). However, there is a connection between the choice of $w_{h, \thresh}$ and the \emph{dual problem} of our optimal transport problem with storage fee. The authors have shown in \cite{BansilKitagawa19a} that a natural dual problem for \eqref{eqn: monge ver} is to maximize 
\begin{align*}
\R^N\ni \psi\mapsto -\int_X \max_{i}(-c(x, y_i)-\psi^i)d\mu(x)-F^*(\psi)
\end{align*}
where $F^*$ is the Legendre transform of $F$. This function is convex, and using \cite[Theorem 1,1]{KitagawaMerigotThibert19}, formally the first order condition for a maximum reads $G(\psi)\in \subdiff{F^*}{\psi}$, or equivalently $\psi\in \subdiff{F}{G(\psi)}$. Under mild conditions, this first order condition actually characterizes optimality, see \cite[Theorem 4.7]{BansilKitagawa19a}. This choice of $w_{h, \thresh}$ is exactly what guarantees that a root of $w_{h, \thresh}-w$ satisfies this first order condition when $F=F_{w, h, \thresh}$ (see Proposition \ref{prop: constructing solutions}).
\end{rmk}

In what follows, it will be possible in theory to obtain the exact dependence of constants on various quantities involving the storage fee function, cost function, domain, and the density of the source measure by tracing these bounds through the results of \cite{KitagawaMerigotThibert19}. However, we are most interested in the dependencies on the parameters $h$, and $\thresh$, thus in the interest of brevity we will introduce the following terminology. The constants below are the same as those introduced in \cite[Remark 4.1]{KitagawaMerigotThibert19}.
\begin{defin}\label{def: universal}
	Suppose $c$ satisfies \eqref{Reg} and \eqref{Twist}, $X$ is a bounded set, $c$-convex with respect to $Y$, $\mu=\rho dx$ for some density $\rho\in C^{0, \alpha}(X)$ for some $\alpha\in (0, 1]$, and $\spt\mu\subset X$. Then we will say that a positive, finite constant is \emph{universal} if it has bounds away from zero and infinity depending only on the following quantities: $\alpha$, $n$, $N$, $\norm{\rho}_{C^{0, \alpha}(X)}$, $\mathcal{H}^{n-1}(\partial X)$, $\max_{i\in \{1,\ldots, N\}}\norm{c(\cdot, y_i)}_{C^2(X)}$, and
	\begin{align*}
		\eps_{\mathrm{tw}}&:=\min_{x\in X} \min_{i, j\in \{1, \ldots, N\}, i \neq j}
		\norm{\nabla_xc(x,y_i) - \nabla_xc(x,y_j)},\\ 
		C_\nabla &:=
		\max_{x\in X, i\in \{1, \ldots, N\}} \norm{\nabla_x c(x,y_i)}\\ 
		C_{\exp}
		&:= \max_{i\in \{1, \ldots, N\}} \max\left\{\norm{\exp_i^c}_{C^{0, 1}(\invcExp{i}{X})},
		\norm{(\exp_i^c)^{-1}}_{C^{0, 1}(X)}\right\},\\
		C_{\mathrm{cond}} &:= \max_{i\in \{1,\ldots, N\}} \max_{p \in \invcExp{i}{X}}
		\mathrm{cond}(D \exp_i^c(p)),\\
		C_{\det} &:= \max_{i\in \{1, \ldots, N\}} \norm{\det(D\exp_i^c)}_{C^{0, 1}(\invcExp{i}{X})},
	\end{align*}
	where $\mathrm{cond}$ is the condition number of a linear transformation.
\end{defin}
\begin{rmk}
Apart from Sections 3 and 4, we have written all estimates to keep as explicit track of $N$ as possible. However, in these two sections doing so is a tedious exercise, in particular it would require careful book-keeping of exactly what norms are being used. We comment that if the collection $\{y_1,\ldots, y_N\}$ is constructed by sampling from a continuous domain $Y$, and $c$ is a cost function on $X\times Y$ satisfying \eqref{Reg}, \eqref{Twist}, and the Ma-Trudinger-Wang condition (along with appropriate convexity conditions on $X$ and $Y$, which we will not detail here), then of the constants introduced in Definition \ref{def: universal}, only $\eps_{\mathrm{tw}}$ will depend on $N$. In particular, if this is the case, the dependencies of all universal constants that arise in the paper (apart from that of $\eps_{\mathrm{tw}}$) can be seen to be polynomial in $N$.
\end{rmk}

Since Algorithm \ref{alg: damped newton} only produces solutions to an approximating problem, we are concerned with how close these solutions might be to the solutions of our original problem. The second and third theorems of our paper show that solutions of \eqref{eqn: monge ver} with the choice $F=F_{\tilde{w}, h, \thresh}$ are in fact close to the solution of the problem with $F_{w}$, if $\tilde{w}$ is close to $w$ and $h, \eps$ are small.

\begin{defin}
	If $A$, $B\subset \R^n$ are Borel sets, the \emph{$\mu$-symmetric distance} between them is
	\begin{align}
	\Delta_\mu(A, B):=\mu(A\Delta B)=\mu((A\setminus B)\cup (B\setminus A)).
	\end{align}
\end{defin}

The following theorem gives quantified closeness for Laguerre cells of the approximating problems to those of the original problem, in terms of the $\mu$-symmetric distance.

\begin{thm}\label{thm: symmetric convergence}
	Suppose $c$ satisfies \eqref{Reg} and \eqref{Twist}, and $\mu$ is absolutely continuous. %
	Also suppose $h>0$, $\thresh\in (0, \frac{1}{2N})$, and ${w}\in \R^N$ with $\sum_{i=1}^N {w}^i\geq 1$, ${w}^i\geq 0$. Then if $\psi_{h, \thresh}\in \mathcal{K}^\thresh$ and $(T, \weightvect)$ is a pair minimizing \eqref{eqn: monge ver} with the storage fee function $F_{w}$,
	\begin{align}\label{eqn: G bound}
	\norm{G(\psi_{h, \thresh}) -\weightvect }_1 \leq 2(N\thresh +  \norm{w_{h, \thresh}(\psi_{h, \thresh}) - {w}}_1 + 2N\sqrt{2C_Lh})
	\end{align}
	and
	\begin{align}\label{eqn: mu diff bound}
	\sum_{i=1}^N \Delta_\mu({\Lag_{i}(\psi_{h, \thresh})}, {T^{-1}(\{y_i\})}) \leq 8N(N\thresh +  \norm{w_{h, \thresh}(\psi_{h, \thresh}) - {w}}_1 + 2N\sqrt{2C_Lh}),
	\end{align}
	where $C_L>0$ is the universal constant from Lemma \ref{lem: strong convex C}.
\end{thm}
In view of Proposition \ref{prop: constructing solutions} below, the above Theorem \ref{thm: symmetric convergence} implies the following. Suppose $w$, $\tilde{w}\in \R^N$, and $(T_{h, \thresh}, \weightvect_{h, \thresh})$ and $(T, \weightvect)$ are minimizers for \eqref{eqn: monge ver} with storage functions $F_{\tilde{w}, h, \thresh}$ and $F_{{w}}$ respectively. By \cite[Proposition 3.5 and Theorem 4.7]{BansilKitagawa19a}, there exists a vector $\psi_{h, \thresh}$  such that $T^{-1}_{h, \thresh}(\{y_i\})=\Lag_i(\psi_{h, \thresh})$ up to sets of zero $\mu$ measure.  By the uniqueness statement of Proposition \ref{prop: constructing solutions}, we see that $w_{h, \thresh}(\psi_{h, \thresh}) =\tilde{w}$, hence the above theorem shows the $\mu$-symmetric distance between $T^{-1}_{h, \thresh}(\{y_i\})$ and $T^{-1}(\{y_i\})$ is controlled by $h$, $\thresh$, and $\norm{w_{h, \thresh}(\psi_{h, \thresh}) - {w}}_1$ (recall this last term is the error term from Algorithm \ref{alg: damped newton}).

The final theorem below shows that when the Laguerre cell associated to the problem with $h=0=\thresh$ has nonzero Lebesgue measure, the above closeness can be measured in the Hausdorff distance. Before stating this result, we recall the following definition.
\begin{defin}
	If $1\leq q\leq \infty$, a probability measure $\mu$ on $X$ satisfies a \emph{$(q, 1)$-Poincar\'e-Wirtinger inequality}  if there is a constant $\Cpw>0$ such that for any $f\in C^1(X)$,
	\begin{align*}
	\norm{f-\int_Xfd\mu}_{L^q(\mu)}\leq \Cpw \norm{\nabla f}_{L^1(\mu)}.
	\end{align*}
	We will say ``$\mu$ satisfies a \emph{$(q, 1)$-PW inequality}''.
\end{defin}
\begin{rmk}\label{rmk: PW}
	Recall that some kind of connectedness condition on $\spt \mu$ is necessary in order to obtain invertibility of the derivative of the map $G$ in nontrivial directions (see the discussion immediately preceding \cite[Definition 1.3]{KitagawaMerigotThibert19}), and a Poincar{\'e}-Wirtinger inequality can be viewed as a quantitatively strengthened version of connectivity which is sufficient for our purposes.
	
	If $\rho$ is bounded away from zero on $\spt\rho$ and the support is connected, it satisfies a $(\frac{n}{n-1}, 1)$-PW inequality, by scaling $q=\frac{n}{n-1}$ is the largest possible value of $q$. We will only use the case of $q>1$ in order to obtain quantitative bounds on the Hausdorff convergence of Laguerre cells, namely for Theorem \ref{thm: hausdorff convergence}. We also remark that in Theorem \ref{thm: hausdorff convergence}, we can make do with $q=1$ if all of the Laguerre cells of the limit problem have nonzero measure. Below, $d_\H$ is the Hausdorff distance between subsets of $\R^n$.
\end{rmk}
\begin{thm}\label{thm: hausdorff convergence}
	Suppose $c$ and $\mu$ satisfy the same conditions as Theorem \ref{thm: linear convergence}, and $\mu$ satisfies a $(q, 1)$-PW inequality for some $q\geq 1$. Also suppose $h>0$, $\thresh\in (0, \frac{1}{2N})$, and ${w}\in \R^N$ with $\sum_{i=1}^N {w}^i>1$, ${w}^i\geq 0$, and $(T, \weightvect)$ is a pair minimizing \eqref{eqn: monge ver} with the storage fee function $F_{w}$, and $\psi\in \R^N$ is such that $T_\psi=T$ $\mu$-a.e.. 
	\begin{enumerate}
		\item If $\{h_k\}_{k=1}^\infty$, $\{\thresh_k\}_{k=1}^\infty\subset\R_{>0}$, $\{\psi_k\}_{k=1}^\infty$, $\psi_k\in\mathcal{K}^{\thresh_k}$ are sequences such that $w_{h_k, \thresh_k}(\psi_k)\to w$, $h_k\searrow 0$, $\thresh_k\searrow 0$ as $k\to \infty$, and $\L(\Lag_i(\psi))>0$, then
		\begin{align*}
		\lim_{k\to 0}d_\H({\Lag_{i}(\psi_k)}, \Lag_i(\psi) )=0.
		\end{align*}
		\item If $q>1$, $\psi_{h, \thresh}\in \mathcal{K}^\thresh$, there are universal constants $C_1$, $C_2>0$ such that,
		\begin{align*}
		d_\H({\Lag_{i}(\psi_{ h, \thresh})}, \Lag_i(\psi) )^n \leq 
		&\frac{C_1\Cpw N^{5} q(N\thresh + \norm{w_{h, \thresh}(\psi_{h, \thresh}) - w}_1 + 2N\sqrt{2C_Lh})}{\thresh^{1/q}(q-1)\left(\arccos(1-C_2\L(\Lag_i(\psi))^2)\right)^{n-1}},
		\end{align*}
		as long as
		\begin{align}\label{eqn: constraint}
		\frac{N^{5}\Cdiff C_\nabla\Cpw  q(N\thresh +  \norm{w_{h, \thresh}(\psi_{h, \thresh}) - w}_1 + 2N\sqrt{2C_Lh})}{\thresh^{1/q} (q-1)} < \L(\Lag_i(\psi))
		\end{align}
		where $\Cdiff$ and $C_L$ are the universal constants defined in \cite[Lemma 3.4]{BansilKitagawa20a} and Lemma \ref{lem: strong convex C} respectively.
\end{enumerate}
\end{thm}
\subsection{Outline of the paper}
In Section \ref{section: properties of w} we give some useful properties of the mapping $w_{h, \thresh}$. In Section \ref{section: convergence of algorithm}, we prove Theorem \ref{thm: linear convergence} on the convergence rate of our Algorithm \ref{alg: damped newton}. We also give a crude estimate on the number of iterations necessary to get within a desired error in terms of the parameters $h$, $\thresh$, and $N$. In section \ref{section: Stability of Laguerre Cells} we prove Theorem \ref{thm: symmetric convergence} and \ref{thm: hausdorff convergence} on the convergence of the Laguerre cells. In Section \ref{section: numerical examples} we present some numerical examples. These examples will include a comparison of performance with the algorithm from \cite{KitagawaMerigotThibert19}, and cases which are outside of the scope of this previous work. Appendix \ref{appendix: strong convexity} contains a short result on strong convexity of the transport cost as a function of the dual variables $\psi$, needed for Theorem \ref{thm: hausdorff convergence}.

\section{Properties of the mapping $w_{h, \thresh}$}\label{section: properties of w}
In this section, we gather some properties and estimates on the mapping $w_{h, \thresh}$ which will be crucial in the proofs of all of our main theorems. For the remainder of the paper, we assume that $c$ satisfies \eqref{Reg}, \eqref{Twist}, and $\mu$ is absolutely continuous. For this section and the following, we also assume $c$ satisfies \eqref{QC}, $\mu=\rho dx$ for some density $\rho\in C^{0, \alpha}(X)$, for some $\alpha\in (0, 1]$, and $X$ is a bounded set, $c$-convex with respect to $Y$ such that $\spt\mu\subset X$.

\subsection{Solutions of the approximating problem with $F_{w, h, \thresh}$}
We will begin by justifying the remarks following Algorithm \ref{alg: damped newton}.
\begin{defin}
	The \emph{subdifferential} of a convex function $F: \R^N\to\R\cup\{+\infty\}$ at any point $x$ is defined by the set
	\begin{align*}
		\subdiff{F}{x}:=\{p\in \R^N\mid F(y)\geq F(x)+\inner{p}{y-x},\ \forall y\in \R^N\}.
	\end{align*}
\end{defin}
\begin{prop}\label{prop: constructing solutions}
	Fix $h$, $\thresh>0$ and $w\in \R^N$ with $w^i\geq 0$, $\sum_{i=1}^N w^i\geq 1$. Then if $\psi\in \R^N$ is such that $w_{h, \thresh}(\psi)=w$, the pair $(T_{\psi}, G(\psi))$ is the unique solution to the minimization problem \eqref{eqn: monge ver} with storage fee function given by $F_{w, h, \thresh}$ (with $T_{\psi}$ defined as in Remark \ref{rmk: brenier solutions}).
\end{prop}
\begin{proof}
	We first calculate for any $t_0\geq 0$ and $t\in (\thresh, t_0+\thresh)$, $\displaystyle\frac{d}{dt}\approxfee_{t_0, h}(t-\thresh) = h \frac{2(t - \thresh)-t_0}{2\sqrt{(t - \thresh)(t_0-t + \thresh)}}$. Thus for any $t$ and $t_1\geq 0$ if we take the choice 
	\begin{align*}
	t_0 = 2(t - \thresh) \( 1+(\frac{t_1}{h})^2 - \frac{t_1}{h}\sqrt{1+(\frac{t_1}{h})^2} \)=(t-\thresh)g(\frac{t_1}{h})
	\end{align*}
	we obtain
	\begin{align*}
	\frac{d}{dt}\approxfee_{t_0, h}(t-\thresh)
	&= h\frac{2(t - \thresh)-(2(t - \thresh)( 1+(\frac{t_1}{h})^2 - \frac{t_1}{h}\sqrt{1+(\frac{t_1}{h})^2} ))}{2\sqrt{(t - \thresh)((2(t - \thresh)( 1+(\frac{t_1}{h})^2 - \frac{t_1}{h}\sqrt{1+(\frac{t_1}{h})^2} ))-(t - \thresh))}} \\
	%&= h\frac{2(t - \thresh)\(-(\frac{t_1}{h})^2 + \frac{t_1}{h}\sqrt{1+(\frac{t_1}{h})^2} \)}{2(t - \thresh)\sqrt{(2( 1+(\frac{t_1}{h})^2 - \frac{t_1}{h}\sqrt{1+(\frac{t_1}{h})^2} ))-1}} \\
	%&= h\frac{t_1}{h} \frac{\sqrt{1+(\frac{t_1}{h})^2} -\frac{t_1}{h}}{\sqrt{ 1+(\frac{t_1}{h})^2 - 2\frac{t_1}{h}\sqrt{1+(\frac{t_1}{h})^2} + (\frac{t_1}{h})^2}} \\
	&= t_1 \frac{\sqrt{1+(\frac{t_1}{h})^2} -\frac{t_1}{h}}
	{\sqrt{(\sqrt{1+(\frac{t_1}{h})^2} - \frac{t_1}{h})^2}} 
	= t_1.
	\end{align*}
	Thus, taking $t=G(\psi)$ and $t_1=\psi^i$, $t_0=(w_{h, \thresh}(\psi))^i$ for each $i$ in the calculation above, we see that if $w_{h, \thresh}(\psi)=w$, we will have $\psi\in \subdiff{F_{w, h, \thresh}}{G(\psi)}$. Since $F_{w, h, \thresh}$ is a proper, convex function that is $+\infty$ outside the set $\weightvectset$, by \cite[Theorem 4.7]{BansilKitagawa19a} we obtain that the pair $(T_\psi, G(\psi))$ is the unique minimizing pair in the problem \eqref{eqn: monge ver} with storage fee function $F_{w, h, \thresh}$.
\end{proof}

\subsection{Estimates on $w_{h, \thresh}$}
Next we will obtain invertibility of $Dw_{h, \thresh}$ on the set $\normset{w, h, \thresh}$. This normalization will be critical in obtaining the necessary estimates to justify convergence of our Newton algorithm. For the remainder of this section and the following Section \ref{section: convergence of algorithm}, we will not be as explicit in terms of the dependence of various quantities on $N$. Related to this, for any vector valued map $\Phi: \Omega\to \R^N$ on any domain $\Omega\subset \R^N$, we will write associated $\alpha$-H\"older seminorms as 
\begin{align*}
 [\Phi]_{C^{0, \alpha}(\overline \Omega)}:&=\sup_{x\neq y\in\Omega}\frac{\norm{\Phi(x)-\Phi(y)}}{\norm{x-y}^\alpha}\leq \sqrt{N}\max_{1\leq i\leq N}\sup_{x\neq y\in\Omega}\frac{\abs{\Phi^i(x)-\Phi^i(y)}}{\norm{x-y}^\alpha}\\
 [D\Phi]_{C^{0, \alpha}(\overline \Omega)}:&=\sup_{x\neq y\in \Omega}\frac{\norm{D\Phi(x)-D\Phi(y)}}{\norm{x-y}^\alpha}\leq N\max_{1\leq i, j\leq N}\sup_{x\neq y\in\Omega}\frac{\abs{D_j\Phi^i(x)-D_j\Phi^i(y)}}{\norm{x-y}^\alpha},
\end{align*}
and
\begin{align*}
 \norm{\Phi}_{C^{1}( \Omega)}:&=\sup_{x\in {\Omega}}\norm{\Phi(x)}+\sup_{x\in {\Omega}} \norm{D\Phi(x)}, \quad
 \norm{\Phi}_{C^{1, \alpha}(\overline \Omega)}:= \norm{\Phi}_{C^{1}( \Omega)}+[D\Phi]_{C^{0, \alpha}(\overline \Omega)}
\end{align*}
where $\norm{D\Phi(x)}$ is the operator norm. 
In particular, for universal constants $C>0$ (that only depend on $N$) we obtain $\norm{\Phi(\psi_1)-\Phi(\psi_2)}\leq C [\Phi]_{C^{0, \alpha}(\overline \Omega)}\norm{\psi_1-\psi_2}^\alpha$, and similar for $D\Phi$.

\begin{prop}\label{prop: Dw estimates}
	Fix $h>0$, $\thresh\in (0, \frac{1}{2N})$, $\epsilon_0>0$, and $w\in \R^N$ with $\sum_{i=1}^Nw^i \geq 1$, $w^i\geq 0$, and suppose $c$, $X$, and $\mu$ satisfy the same conditions as Theorem \ref{thm: linear convergence}. Then %
	\begin{enumerate}
		\item $\normset{w, h, \thresh}$ is bounded and nonempty.
		\item $w_{h, \thresh}$ is differentiable on $\mathcal{K}^\thresh$.
		\item $Dw_{h, \thresh}(\psi)$ is invertible whenever $\psi\in\normset{w, h, \thresh}\cap\mathcal{W}^{\epsilon_0}$.
	\end{enumerate}
	Moreover if $h\leq 1$, there exists a universal constant $C>0$ such that
	\begin{align}
	\diam(\normset{w, h, \thresh})&\leq C\thresh^{-\frac{1}{2}}\label{eqn: diameter bound}\\
		\norm{w_{h, \thresh}}_{C^{1, \alpha}(\overline{\normset{w, h, \thresh}})}& =: L \leq C \max\(h^{-2}\thresh^{-2}, h^{-3}\thresh^{-\frac{1}{2}}\), \label{eqn: lipschitz bound w}\\
		\sup_{\psi\in \normset{w, h, \thresh}\cap\mathcal{W}^{\epsilon_0}}\norm{Dw_{h, \thresh}(\psi)^{-1}}& =: \kappa^{-1} \leq C \epsilon_0^{-1} h^{-6}\thresh^{-\frac{3}{2}}.\label{eqn: inverse Dw bound}
	\end{align}
\end{prop}
\begin{proof}[Proof of Proposition \ref{prop: Dw estimates}]
	Throughout the proof, $C>0$ will denote a universal constant whose value may change from line to line.
	
	We first calculate
	\begin{align*}
		g'(t) &= 2 \(2t - \sqrt{1+t^2} - \frac{t^2}{\sqrt{1+t^2}} \)
		=\frac{2(2t\sqrt{1+t^2}-1-t^2-t^2)}{\sqrt{1+t^2}}=-\frac{2(t-\sqrt{1+t^2})^2}{\sqrt{1+t^2}}<0.
	\end{align*}
	In particular, $g$ is continuous and strictly decreasing on $\R$, and it is easily seen that $\lim_{t\to-\infty}g=+\infty$ and $\lim_{t\to +\infty}=1$.  Now notice there exists at least one vector $\psi\in {\mathcal{K}^\thresh}$, for such a $\psi$, $G^i(\psi)-\thresh>0$ for all $i$.  Since adding a multiple of $\onevect$ to $\psi$ does not change the value of $G(\psi)$ and $\sum_{i=1}^N(G^i(\psi)-\thresh)<1\leq \sum_{i=1}^Nw^i$, we can see there exists some $r\in \R$ such that $\sum_{i=1}^Nw^i_{h, \thresh}(\psi+r\onevect)=\sum_{i=1}^N(G^i(\psi+r\onevect)-\thresh)g(\frac{\psi^i+r}{h})=\sum_{i=1}^Nw^i$, i.e. $\normset{w, h, \thresh}$ is nonempty.
	
	Next we show boundedness of $\normset{w, h, \thresh}$. If $\psi\in \normset{w, h, \thresh}$, we calculate
	\begin{align*}
	\sum_{i=1}^N  w^i 
	= \sum_{i=1}^N (G^i(\psi) - \thresh)g(\frac{\psi^i}{h}) 
	\leq \sum_{i=1}^N (G^i(\psi) - \thresh)\max_j g(\frac{\psi^j}{h}) 
	= \max_j g(\frac{\psi^j}{h}) (1 - N\thresh).
	\end{align*}

	Hence $\max_j g(\frac{\psi^j}{h}) \geq \frac{\sum_{i=1}^N w^i }{1 - N\thresh} \geq \frac{1}{1 - N\thresh}> 1$. In particular we must have an upper bound on some component $\psi^k$, i.e. $\psi^k \leq \ti{M}_1$  where $\ti{M}_1 := hg^{-1}(\frac{1}{1 - N\thresh}) <+\infty$. Now since $X$ is compact, there exist constants $M_1$ and $m_1$ such that $m_1< c(\cdot, y_i) < M_1$ for all $i\in \{1, \ldots, N\}$. If, for any $i$, $\psi^i >  \ti{M}_1 + M_1-m_1$ then we would have $\Lag_i(\psi)=\emptyset$, contradicting $\psi\in \mathcal{K}^\thresh$. %
	
	A similar calculation yields the bound 
	$
	\min_j g(\frac{\psi^j}{h}) \leq\frac{\sum_{i=1}^N w^i}{1 - N\thresh}
	\leq \frac{N}{1-N\thresh}\leq 2N 
	$, 
	thus by an analogous argument we obtain the uniform bounds	
\begin{align}
 \ti{m}&\leq \psi^i\leq \ti{M},\quad\forall \psi \in \normset{w, h, \thresh},\ i\in \{1,\ldots, N\},\notag\\
 \tilde{M}:&=\ti{M}_1 + M_1-m_1=hg^{-1}(\frac{1}{1 - N\thresh}) + M_1-m_1> 0\notag\\
 \tilde{m}:&=\ti{M}_2-M_1+m_1:=hg^{-1}(2N)-M_1+m_1<0.\label{eqn: coordinate bounds}
\end{align}

	We now calculate bounds on $\ti{M}$ and $\ti{m}$ in terms of $N$ and $\thresh$. If $g(t)=a$ for some value $a>1$, we find 	
\begin{align*}
 \frac{a}{2}&=1+t^2-t\sqrt{1+t^2}=1+t(t-\sqrt{1+t^2})=1+t\(\frac{-1}{t+\sqrt{1+t^2}}\)=\frac{\sqrt{1+t^2}}{t+\sqrt{1+t^2}}
\end{align*}
hence
\begin{align}\label{eqn: g inverse}
 (1-\frac{a}{2})\sqrt{1+t^2}=\frac{at}{2}\implies (1-\frac{a}{2})^2=t^2(\frac{a^2}{4}-(1-\frac{a}{2})^2)\implies t^2=\frac{(2-a)^2}{4a-4}.
\end{align}
Now if $a=\frac{1}{1-N\thresh} < 2$, we have $t=g^{-1}(a)>0$, hence by \eqref{eqn: g inverse} above,
\begin{align}
 0< \ti{M}\leq C\(1+h\frac{2-\frac{1}{1-N\thresh}}{2\sqrt{\frac{1}{1-N\thresh}-1}}\)=C\(1+h\frac{1}{2\sqrt{N\thresh(1-N\thresh)}}\)\leq \frac{C}{\sqrt{2N\thresh}},\label{eqn: tilde M_1 est}
\end{align}
where we have used that $\thresh<\frac{1}{2N}$. 
Similarly, for $a=2N>2$, $t=g^{-1}(a)<0$ hence using \eqref{eqn: g inverse} again yields
\begin{align}\label{eqn: tilde M_2 est}
0> \ti{m}&=-C(1+h\frac{2N-2}{2\sqrt{2N-1}})\geq -C\(1+\frac{hN}{\sqrt{N}}\)=-C\sqrt{N}.
\end{align}
Combining this with \eqref{eqn: tilde M_1 est} immediately gives \eqref{eqn: diameter bound}.
	
	We will also have use for some estimates on $g$ and $g'$.  We calculate,
\begin{align*}
g'(\frac{\tilde{M}}{h})&=-\frac{2(\frac{\tilde{M}}{h}-\sqrt{1+\(\frac{\tilde{M}}{h}\)^2})^2}{\sqrt{1+\(\frac{\tilde{M}}{h}\)^2}}=-\frac{2(\tilde{M}-\sqrt{h^2+\tilde{M}^2})^2}{h\sqrt{h^2+\tilde{M}^2}}\\
 &=-\frac{2h^3}{\sqrt{h^2+\tilde{M}^2}(\tilde{M}+\sqrt{h^2+\tilde{M}^2})^2}\leq -\frac{h^3}{2(h^2+\tilde{M}^2)^{3/2}}\leq -\frac{h^3}{2(\frac{2N\thresh h^2+C}{2N\thresh})^{3/2}}\leq -Ch^3N^{\frac{3}{2}}\thresh^{\frac{3}{2}}
 \end{align*}
where we have used \eqref{eqn: tilde M_1 est} in the last line. At the same time,
\begin{align*}
 g'(\frac{\tilde{m}}{h})=-\frac{2(\tilde{m}-\sqrt{h^2+\tilde{m}^2})^2}{h\sqrt{h^2+\tilde{m}^2}}\geq -\frac{CN}{h^2},
\end{align*}
since $g'$ is decreasing and negative, we have for any $\psi\in \normset{w, h, \thresh}$ and index $i$, the estimates
\begin{align}\label{eqn: g' estimates}
 Ch^3N^{\frac{3}{2}}\leq \abs{ g'(\frac{\psi^i}{h})}\leq \frac{CN}{h^2}.
\end{align}
Additionally using \eqref{eqn: tilde M_2 est} and that $h\leq 1$,  for any $\psi\in \normset{w, h, \thresh}$ and index $i$ we have (recall $\ti m$ could be negative here)
\begin{align}
1&\leq g(\frac{\psi^i}{h})\leq g(\frac{\tilde{m}}{h})=2\(1+\(\frac{\tilde{m}}{h}\)^2-\frac{\tilde{m}}{h}\sqrt{1+\(\frac{\tilde{m}}{h}\)^2}\)\notag\\
&=\frac{2\sqrt{h^2+\tilde{m}^2}}{h^2}\(\sqrt{h^2+\tilde{m}^2}-\tilde{m}\)\leq \frac{C\tilde{m}^2}{h^2}\leq \frac{CN}{h^2}.\label{eqn: g estimate}
\end{align}

	Under the current assumptions, we see by \cite[Theorem 4.1]{KitagawaMerigotThibert19} that $G$ is uniformly $C^{1, \alpha}$ on $\normset{w, h, \thresh}\subset \mathcal{K}^\thresh$. We then calculate the derivative of $w_{h, \thresh}$ as
	\begin{align}
		Dw_{h, \thresh}(\psi) &= \diag(g(\frac{\psi^i}{h}))DG(\psi) + \frac{1}{h}\diag((G^i(\psi) - \thresh) g'(\frac{\psi^i}{h}))\notag\\
		&=\diag(g(\frac{\psi^i}{h}))\(\frac{1}{h}\diag\(\frac{(G^i(\psi) - \thresh) g'(\frac{\psi^i}{h})}{g(\frac{\psi^i}{h})}\)+DG(\psi)\)\label{eqn: Dw formula}
	\end{align}
	where $\diag$ of a vector in $\R^N$ is the $N\times N$ diagonal matrix with the entries of the vector on the diagonal. Since $g\geq 1$ on $\R$, we see $\diag(g(\frac{\psi^i}{h}))$ is invertible with all eigenvalues larger than $1$. For any unit vector $V\in \R^N$ we have
	\begin{align*}
	&\inner{\frac{1}{h}\diag\(\frac{ (G^i(\psi) - \thresh) g'(\frac{\psi^i}{h})}{g(\frac{\psi^i}{h})}\)V}{V}+\inner{DG(\psi)V}{V}\\
	&=\frac{1}{h}\sum_{i=1}^N\frac{(G^i(\psi) - \thresh) g'(\frac{\psi^i}{h})}{g(\frac{\psi^i}{h})}(V^i)^2+\inner{DG(\psi)V}{V}=:A+B.
	\end{align*}
	By \cite[Theorem 1.1 and 1.3]{KitagawaMerigotThibert19}, $DG$ is symmetric, every off diagonal entry is nonnegative, and each row sums to zero, hence $B\leq 0$. We also calculate	
\begin{align*}
 A&\leq \frac{1}{h}\max_j \frac{(G^j(\psi) - \thresh) g'(\frac{\psi^j}{h})}{g(\frac{\psi^j}{h})}\sum_{i=1}^N(V^i)^2=\frac{1}{h}\max_j \frac{(G^j(\psi) - \thresh) g'(\frac{\psi^j}{h})}{g(\frac{\psi^j}{h})}\\
 &= \frac{1}{h}\max_j \frac{w^j_{h, \thresh}(\psi) g'(\frac{\psi^j}{h})}{g(\frac{\psi^j}{h})^2}\leq\frac{-m_{h, \thresh} \epsilon_0}{hg(\frac{\tilde{m}}{h})^2}\leq -C \epsilon_0 h^6N^{-\frac{1}{2}}\thresh^{\frac{3}{2}}
\end{align*}
	where we use \eqref{eqn: g estimate} and that $\psi\in\mathcal{W}^{\epsilon_0}$, hence $Dw_{h, \thresh}(\psi)$ is invertible and we obtain \eqref{eqn: inverse Dw bound}.
	
	Finally, since $\normset{w, h, \thresh}$ is bounded by above and $g'$ is clearly a $C^1$ function on $\R$, we can again use \cite[Theorem 4.1]{KitagawaMerigotThibert19} to conclude that $w_{h, \thresh}$ is actually $C^{1, \alpha}$ on $\normset{w, h, \thresh}$. The only thing left is to verify the dependencies of $L>0$ from \eqref{eqn: lipschitz bound w}.  Since $g$ is decreasing on $\R$, by \eqref{eqn: g estimate} we immediately see that $\norm{w_{h, \thresh}}_{L^\infty(\normset{w, h, \thresh})}\leq \frac{CN}{h^2}$ . Also calculating 
 using \eqref{eqn: g' estimates}, \eqref{eqn: Dw formula}, \eqref{eqn: g estimate}, and that $\norm{G}_{C^1(\overline{\mathcal{K}^{\thresh}})}\leq CN$ from \cite[Theorem 1.3]{KitagawaMerigotThibert19}, we see that $\norm{w_{h, \thresh}}_{C^1(\normset{w, h, \thresh})}\leq C(N^2h^{-2}+Nh^{-2})\leq \frac{CN^2}{h^2}$.

For the remainder of the proof, we will not keep explicit track of the dependencies on $N$. Finally, note that
\begin{align}
 [Dw_{h, \thresh}]_{C^{0, \alpha}}&\leq C\bigg(\norm{g(\frac{\cdot}{h})}_{L^\infty}[DG]_{C^{0, \alpha}}+[g(\frac{\cdot}{h})]_{C^{0, \alpha}}\norm{DG}_{L^\infty}\notag\\
 &+\frac{1}{h}([G-\epsilon\onevect]_{C^{0, \alpha}}\norm{g'(\frac{\cdot}{h})}_{L^\infty}+[G-\epsilon\onevect]_{L^\infty}[g'(\frac{\cdot}{h})]_{C^{0, \alpha}})\bigg)\notag\\
 &\leq C\bigg(\norm{g(\frac{\cdot}{h})}_{L^\infty}[DG]_{C^{0, \alpha}}+\diam(\normset{w, h, \thresh})\norm{g'(\frac{\cdot}{h})}_{L^\infty}\norm{DG}_{L^\infty}\notag\\
 &+\frac{\diam(\normset{w, h, \thresh})}{h}(\norm{DG}_{L^\infty}\norm{g'(\frac{\cdot}{h})}_{L^\infty}+[G-\epsilon\onevect]_{L^\infty}[g'(\frac{\cdot}{h})]_{C^{0, 1}})\bigg)\label{eqn: Dw holder expression}
\end{align}
where all norms and seminorms of $g$ and $g'$ are over $[\ti{m}, \ti{M}]$ and the remainder over $\normset{w, h, \thresh}$.

Fixing an index $i$, for any $\psi_1\neq \psi_2\in \normset{w, h, \thresh}$ we have
\begin{align}
 \abs{g'(\frac{\psi_1^i}{h})-g'(\frac{\psi_2^i}{h})}
 &\leq \sup_{t\in [\ti{m}, \ti{M}]}\abs{g''(\frac{t}{h})}\abs{\frac{\psi_1^i}{h}-\frac{\psi_2^i}{h}}\leq \frac{C\norm{\psi_1-\psi_2}}{h},\label{eqn: 2 for Dw holder}
\end{align}
since by direct computation
we see
\begin{align*}
g''(t) 
= \frac{-4t^3 + 4(1+t^2)^{3/2} - 6t}{(1+t^2)^{3/2}} 
= 4 - 2\frac{2t^3 + 3t}{(1+t^2)^{3/2}}
= 4 - 4\frac{t}{(1+t^2)^{1/2}} - 2\frac{t}{(1+t^2)^{3/2}}
\end{align*}
and so
\begin{align*}
\abs{g''(t)} 
\leq 4 + 4 \abs{\frac{t}{(1+t^2)^{1/2}}} + 2 \abs{\frac{t}{(1+t^2)^{3/2}}}
\leq 4 + 4 + 2 \min(\abs {t}, {\abs{t}}^{-2})
\leq 10. 
\end{align*}
At the same time using \eqref{eqn: g' estimates},
\begin{align}
 \abs{g(\frac{\psi_1^i}{h})-g(\frac{\psi_2^i}{h})}&\leq \sup_{t\in [\ti{m}, \ti{M}]}\abs{g'(\frac{t}{h})}^2\abs{\frac{\psi_1^i}{h}-\frac{\psi_2^i}{h}}\leq \frac{C}{h^5}\norm{\psi_1-\psi_2}.\label{eqn: 1 for Dw holder}
\end{align}
Finally, carefully tracing through the proofs leading to \cite[Theorem 4.1]{KitagawaMerigotThibert19} yields that 
\begin{align}\label{eqn: dg holder est}
 	[DG]_{C^{0, \alpha}(\conj{\mathcal{K}^\thresh})}\leq \frac{C}{\thresh^2},
 \end{align}
combining with  \eqref{eqn: g' estimates}, \eqref{eqn: g estimate}, \eqref{eqn: 1 for Dw holder}, \eqref{eqn: 2 for Dw holder}, and that $\norm{G}_{C^1(\conj{\mathcal{K}^\thresh})}\leq CN$  in \eqref{eqn: Dw holder expression} we obtain
\begin{align*}
 [Dw_{h, \thresh}]_{C^{0, \alpha}(\normset{w, h, \thresh})}&\leq C\max\(h^{-2}\thresh^{-2}, h^{-3}\thresh^{-\frac{1}{2}}\).
\end{align*}	
\end{proof}

\section{Convergence of Algorithm \ref{alg: damped newton}}\label{section: convergence of algorithm}
Here we provide the proof of our first main theorem, on global linear and locally superlinear convergence of Algorithm \ref{alg: damped newton}. We remark that the proof below also shows that $\normset{w, h, \thresh}$ is locally a $C^1$ manifold of codimension $1$ in $\R^n$. Again, we will not track explicit dependencies on $N$.

\begin{prop}\label{prop: alg error bounds}
There is a function $r\in C^{1, \alpha}(\overline{\mathcal{K}^\thresh})$ such that for any $\psi\in \R^N$, $r(\psi)$ is the unique number such that $\pi(\psi):=\psi-r(\psi)\onevect\in \normset{w, h, \thresh}$. Moreover, for some universal $C>0$,
\begin{align*}
\norm{D\pi}_{C^{0, \alpha} (\overline{\mathcal{K}^\thresh};\R^N)} &\leq\frac{C}{h^{18}\thresh^{9}}.
\end{align*}
\end{prop}
\begin{proof}
	First we carry out some preliminary analysis. Again, $C>0$ will denote a suitable universal constant throughout the proof. Define $\R^N\times \R\ni (\psi, r)\to \Phi(\psi, r)\in \R$ by
	\begin{align*}
		\Phi(\psi, r)&=\sum_{i=1}^N w_{h, \thresh}^i(\psi-r\onevect)- w^i=\sum_{i=1}^N (G^i(\psi-r\onevect)-\thresh)g(\frac{\psi^i-r}{h})-\sum_{i=1}^N w^i\\
		&=\sum_{i=1}^N (G^i(\psi)-\thresh)g(\frac{\psi^i-r}{h})-\sum_{i=1}^N w^i.
	\end{align*}
	Note for any $\psi\in \R^N$ such that $w_{h, \thresh}^i(\psi)\geq 0$ for all $i\in \{1, \ldots, N\}$, we must have $G^i(\psi)\geq \thresh$, hence $\psi\in \mathcal{K}^\thresh$ for such $\psi$. A quick calculation yields that if $(\psi, r)$ are such that $\psi\in \mathcal{K}^\thresh$ and $\psi-r\onevect\in \normset{w, h, \thresh}$, we have using the calculation immediately preceding \eqref{eqn: g' estimates},
	\begin{align*}
		\frac{\partial}{\partial r}\Phi(\psi, r)%
		&=-\frac{1}{h}\sum_{i=1}^N (G^i(\psi)-\thresh)g'(\frac{\psi^i-r}{h})\geq \frac{CN}{h^3}(1-N\thresh)>0.
	\end{align*}
	Now, the strict monotonicity of $g$ along with the fact that $\sum_{i=1}^Nw^i\geq 1>\sum_{i=1}^N(G^i(\psi)-\epsilon)$ and $g(\R)=(1, \infty)$ implies that for any $\psi\in\R^N$, there exists a unique $r(\psi)\in \R$ such that $\Phi(\psi, r(\psi))=0$, thus the function $\psi\mapsto r(\psi)$ is well-defined. By the above calculation and the implicit function theorem we have that this function $r$ is differentiable near any $\psi\in \mathcal{K}^\thresh$. Differentiating the expression $\Phi(\psi, r(\psi))=0$ with respect to $\psi^j$ at such a $\psi$, we find that
	\begin{align}
		0%
		&=\sum_{i=1}^N \(D_jG^i(\psi)g(\frac{\psi^i-r(\psi)}{h})+(G^i(\psi)-\thresh)g'(\frac{\psi^i-r(\psi)}{h})\frac{\delta^i_j-D_jr(\psi)}{h}\)\notag\\
		\implies D_jr(\psi)&=\frac{\sum_{i=1}^N hD_jG^i(\psi)g(\frac{\psi^i-r(\psi)}{h})+\delta^i_j(G^i(\psi)-\thresh)g'(\frac{\psi^i-r(\psi)}{h})}{\sum_{i=1}^N(G^i(\psi)-\thresh)g'(\frac{\psi^i-r(\psi)}{h})}\notag\\
		&=\frac{(G^j(\psi)-\thresh)g'(\frac{\psi^j-r(\psi)}{h})+h\sum_{i=1}^N D_jG^i(\psi)g(\frac{\psi^i-r(\psi)}{h})}{\sum_{i=1}^N(G^i(\psi)-\thresh)g'(\frac{\psi^i-r(\psi)}{h})}.\label{eqn: grad r}
	\end{align}
	We can see $\norm{Dr}$ is uniformly bounded on $\mathcal{K}^\thresh$: we calculate
	\begin{align}
		\norm{D_jr}_{L^\infty(\mathcal{K}^\thresh)}
		&\leq 1+\frac{\abs{\sum_{i=1}^N  D_jG^i(\psi)g(\frac{\psi^i-r(\psi)}{h})}}{\abs{\frac{1}{h}\sum_{i=1}^N(G^i(\psi)-\thresh)g'(\frac{\psi^i-r(\psi)}{h})}} \notag \\
		&\leq 1+\frac{g(\frac{\ti m}{h})\sum_{i=1}^N \abs{D_jG^i(\psi)}}{\frac{h^3\thresh^{3/2}}{h}(1 - N \thresh)}
		\leq 1+\frac{C(\frac{1}{h^2})}{h^2\thresh^{3/2}(1 - N \thresh)} \leq \frac{C}{h^4\thresh^{\frac{3}{2}}}\label{eqn: Dr bound}
	\end{align}
	where we have used 
	$\norm{G}_{C^1(\overline{\mathcal{K}^\thresh})}\leq C$ from \cite[Theorem 1.3]{KitagawaMerigotThibert19}, 
	\eqref{eqn: g' estimates}, 
	\eqref{eqn: g estimate}, and that $\thresh<\frac{1}{2N}$. %

	Since $\mathcal{K}^\thresh = \bigcap_{i=1}^N (G^i)^{-1}((\thresh, \infty))$, the implicit function theorem combined with \cite[Theorem 5.1]{KitagawaMerigotThibert19} along with the fact that $\partial X$ is locally Lipschitz shows that $\partial \mathcal{K}^\thresh$ is locally Lipschitz. Thus $W^{1, \infty}(\mathcal{K}^\thresh)=C^{0, 1}(\overline {\mathcal{K}^\thresh})$, hence $r$ is uniformly Lipschitz continuous on $\mathcal{K}^\thresh$.

	We will now show a H\"older bound on $Dr$. Note that for each $j$, we can write $D_jr=\frac{H_1}{H_2}$ where $H_1(\psi):= \frac{1}{h}(G^j(\psi)-\thresh)g'(\frac{\psi^j-r(\psi)}{h})+\sum_{i=1}^N  D_jG^i(\psi)g(\frac{\psi^i-r(\psi)}{h})$ belongs to $C^{0, \alpha}(\overline{\mathcal{K}^\thresh})$ (using \cite[Theorem 4.1]{KitagawaMerigotThibert19}) and $H_2(\psi):= \frac{1}{h}\sum_{i=1}^N(G^i(\psi)-\thresh)g'(\frac{\psi^i-r(\psi)}{h})$ belongs to $C^{0, 1}(\overline{\mathcal{K}^\thresh})$, with $H_2\leq -\frac{Ch^3N^{3/2}}{h}(1-N\thresh)<0$ uniformly. Note that
	\begin{align*}
	H_2(\pi(\psi)) 
	&= \frac{1}{h}\sum_{i=1}^N(G^i(\psi - r(\psi)\onevect)-\thresh)g'(\frac{(\psi - r(\psi)\onevect)^i-r(\psi - r(\psi)\onevect)}{h}) \\
	&= \frac{1}{h}\sum_{i=1}^N(G^i(\psi)-\thresh)g'(\frac{(\psi - r(\psi)\onevect)^i}{h}) 
	= H_2(\psi).
	\end{align*}
	Thus for $\psi_1\neq\psi_2\in \mathcal{K}^\thresh$, using \eqref{eqn: g' estimates},
	\begin{align}
	& \abs{D_jr(\psi_1)-D_jr(\psi_2)}=\abs{\frac{H_1(\psi_1)}{H_2(\psi_1)}-\frac{H_1(\psi_2)}{H_2(\psi_2)}}\leq \abs{\frac{H_1(\psi_1)-H_1(\psi_2)}{H_2(\psi_1)}}+\abs{\frac{H_1(\psi_2)(H_2(\psi_2)-H_2(\psi_1))}{H_2(\psi_1)H_2(\psi_2)}}\notag\\
	&= \abs{\frac{H_1(\psi_1)-H_1(\psi_2)}{H_2(\psi_1)}}+\abs{\frac{H_1(\psi_2)(H_2(\pi(\psi_2))-H_2(\pi(\psi_1)))}{H_2(\psi_1)H_2(\psi_2)}}\notag\\
	&\leq \frac{[H_1]_{C^{0, \alpha}(\overline{\mathcal{K}^\thresh})}\norm{\psi_1-\psi_2}^\alpha}{\frac{Ch^3N^{\frac{3}{2}}}{h}(1-N\thresh)}+\frac{\norm{H_1}_{L^\infty(\mathcal{K}^\thresh)}[H_2]_{C^{0, 1}(\overline{\mathcal{K}^\thresh})}\norm{\pi(\psi_2)-\pi(\psi_1)}}{(\frac{Ch^3N^{\frac{3}{2}}}{h}(1-N\thresh))^2}\notag\\
	&\leq C\left( \frac{[H_1]_{C^{0, \alpha}(\overline{\mathcal{K}^\thresh})}}{h^2N^{\frac{3}{2}}(1-N\thresh)}+\frac{\norm{H_1}_{L^\infty(\mathcal{K}^\thresh)}[H_2]_{C^{0, 1}(\overline{\mathcal{K}^\thresh})}\norm{\pi(\psi_2)-\pi(\psi_1)}^{1-\alpha} [\pi]_{C^{0, 1}(\overline{\mathcal{K}^\thresh})}^\alpha } {(h^2N^{\frac{3}{2}}(1-N\thresh))^2} \right) \norm{\psi_1-\psi_2}^\alpha,\label{eqn: Dr holder}
	\end{align}
	hence $D_jr$ is uniformly $C^{0, \alpha}$ on $\mathcal{K}^\thresh$. Our next task will be to estimate ${[Dr]}_{C^{0, \alpha}(\overline{\mathcal{K}^\thresh})}$. In order to do this we estimate each of the terms in the above expression.

	A quick calculation yields
	\begin{align}
	\norm{H_1}_{L^\infty(\mathcal{K}^\thresh)} \leq C(\frac{1}{h^3}+\frac{1}{h^2}) \leq \frac{C}{h^3},\label{eqn: H_1 bound}%
	\end{align}
	and since $\pi(\psi) \in \normset{w, h, \thresh}$, by \eqref{eqn: diameter bound} we have
	\begin{align}
	\norm{\pi(\psi_2)-\pi(\psi_1)}&\leq \diam(\normset{w, h, \thresh})\leq \frac{C}{\thresh^{\frac{1}{2}}}.\label{eqn: pi bound}
	\end{align}
	To estimate $[H_2]_{C^{0, 1}(\overline{\mathcal{K}^\thresh})}$, 
	let $H_{3,i}(\psi) := (G^i(\psi)-\thresh)g'(\frac{\psi^i}{h})$ so that $H_2(\psi) = \frac{1}{h} \sum_i H_{3,i}(\pi(\psi))$.  Just as we estimated the final two terms in \eqref{eqn: Dw holder expression}, we see that $[H_{3,i}]_{C^{0, 1}(\normset{w, h, \thresh})} \leq \frac{C}{h^2}$ %
	by using the bound $\norm{G}_{C^1(\overline{\mathcal{K}^\thresh})}\leq C$ with \eqref{eqn: g' estimates} and \eqref{eqn: 2 for Dw holder}. Furthermore since $\pi(\psi)=\psi-r(\psi)\onevect$,\begin{align}
	[\pi]_{C^{0, 1}(\overline{\mathcal{K}^\thresh})} \leq 1 + N^{1/2}[r]_{C^{0, 1}(\overline{\mathcal{K}^\thresh})} \leq \frac{C}{h^4\thresh^{3/2}}\label{eqn: lipschitz of pi}
	\end{align}%
	by \eqref{eqn: Dr bound}. Hence 
	\begin{align}
	[H_2]_{C^{0, 1}(\overline{\mathcal{K}^\thresh})} 
	\leq \frac{1}{h}\sum_{i=1}^N [H_{3,i} \circ \pi]_{C^{0, 1}(\overline{\mathcal{K}^\thresh})} 
	\leq \frac{1}{h}\sum_{i=1}^N[H_{3,i}]_{C^{0, 1}(\normset{w, h, \thresh})}[\pi]_{C^{0, 1}(\overline{\mathcal{K}^\thresh})}\leq \frac{C}{h^7\thresh^{\frac{3}{2}}}.\label{eqn: H_2 bound}%
	\end{align}
	
	Finally we bound $[H_1]_{C^{0, \alpha}(\overline{\mathcal{K}^\thresh})}$. Let $H_{4,i}(\psi) := D_jG^i(\psi)g(\frac{\psi^i}{h})$ so that $H_1(\psi) =(G^j(\psi)-\thresh)g'(\frac{\psi^j-r(\psi)}{h})+ \sum_i H_{4,i}(\pi(\psi))$. For $\psi_1, \psi_2 \in \overline{\mathcal{K}^\thresh}$ we have
	\begin{align*}
	\abs{H_{4,i}(\pi(\psi_1)) - H_{4,i}(\pi(\psi_2)) }
	%&= \abs{D_jG^i(\psi_1)g(\frac{\pi(\psi_1)^i}{h}) - D_jG^i(\psi_2)g(\frac{\pi(\psi_2)^i}{h})} \\
	&= \abs{(D_jG^i(\psi_1)- D_jG^i(\psi_2))g(\frac{\pi(\psi_1)^i}{h}) - D_jG^i(\psi_2)(g(\frac{\pi(\psi_1)^i}{h})-g(\frac{\pi(\psi_2)^i}{h}))} \\
	&\leq [DG]_{C^{0, \alpha}(\overline{\mathcal{K}^\thresh})}g(\frac{\ti m}{h})\norm{\psi_1-\psi_2}^\alpha + \norm{G}_{C^1(\overline{\mathcal{K}^\thresh})}  \sup_{s\in[\ti m, \ti M]}\abs{g'(\frac{s}{h})}\frac{\norm{\pi(\psi_1)-\pi(\psi_2)}}{h} \\
	&\leq \left( [DG]_{C^{0, \alpha}(\conj{\mathcal{K}^\thresh})}g(\frac{\ti m}{h}) + \frac{C}{h^3}\norm{\pi(\psi_1)-\pi(\psi_2)}^{1-\alpha}[\pi]_{C^{0, 1}(\overline{\mathcal{K}^\thresh})}^\alpha \right) \norm{\psi_1-\psi_2}^\alpha \\
	&\leq C\left( \frac{1}{ h^2\thresh^2} + \frac{1}{h^{3+4\alpha}\thresh^{\frac{1}{2}+\alpha}}   \right) \norm{\psi_1-\psi_2}^\alpha
	\leq \frac{C}{ h^7\thresh^2} \norm{\psi_1-\psi_2}^\alpha	
	\end{align*}
	where we have used \eqref{eqn: g' estimates} to estimate $g'$, \eqref{eqn: g estimate} to estimate $g(\frac{\ti m}{h})$,  \cite[Theorem 1.3]{KitagawaMerigotThibert19} to estimate $\norm{G}_{C^1(\overline{\mathcal{K}^{\thresh}})}$, \eqref{eqn: dg holder est} for $[DG]_{C^{0, \alpha}(\conj{\mathcal{K}^\thresh})}\leq \frac{C}{\thresh^2}$, and \eqref{eqn: pi bound}. Hence we see, using \eqref{eqn: lipschitz of pi}, 
\begin{align*}
&[H_1]_{C^{0, \alpha}(\overline{\mathcal{K}^\thresh})}% \leq C([G^j]_{C^{0, \alpha}(\normset{w, h, \thresh})}\norm{g'(\frac{\cdot}{h})}_{L^\infty([\ti m, \ti M])}+\norm{G^j}_{L^\infty(\overline{\mathcal{K}^\thresh})}[g'(\frac{\cdot}{h})]_{C^{0, \alpha}([\ti m, \ti M])})[\pi]_{C^{0, 1}(\overline{\mathcal{K}^\thresh})}\\
%&+\sum_{i=1}^N [H_{4,i} \circ \pi]_{C^{0, \alpha}(\overline{\mathcal{K}^\thresh})}\\
\leq C\bigg(\diam (\normset{w, h, \thresh})[G^j]_{C^{0, 1}(\normset{w, h, \thresh})}\norm{g'(\frac{\cdot}{h})}_{L^\infty([\ti m, \ti M])}+(\ti M-\ti m)\norm{G^j}_{L^\infty(\overline{\mathcal{K}^\thresh})}[g'(\frac{\cdot}{h})]_{C^{0, 1}([\ti m, \ti M])}\\
&+\sum_{i=1}^N [H_{4,i}]_{C^{0, \alpha}(\normset{w, h, \thresh})} \bigg)[\pi]_{C^{0, 1}(\overline{\mathcal{K}^\thresh})}
\leq C\(\frac{1}{h^2\thresh^{\frac{1}{2}}}+\frac{1}{h\thresh^{\frac{1}{2}}}+\frac{1}{h^7\thresh^2}\)\frac{1}{h^4\thresh^{\frac{3}{2}}}\leq \frac{C}{ h^{11}\thresh^{\frac{7}{2}}}.
\end{align*}
	Putting the above together with \eqref{eqn: Dr holder}, \eqref{eqn: H_1 bound}, \eqref{eqn: pi bound}, and \eqref{eqn: H_2 bound} we get
	\begin{align*}
	[Dr]_{C^{0, \alpha} (\overline{\mathcal{K}^\thresh})} &\leq C\(
	\frac{[H_1]_{C^{0, \alpha}(\overline{\mathcal{K}^\thresh})}}{h^2N^{\frac{3}{2}}(1-N\thresh)}+\frac{\norm{H_1}_{L^\infty(\mathcal{K}^\thresh)}[H_2]_{C^{0, 1}(\overline{\mathcal{K}^\thresh})}\norm{\pi(\psi_2)-\pi(\psi_1)}^{1-\alpha}[\pi]_{C^{0, 1}(\overline{\mathcal{K}^\thresh})}^\alpha } {(h^2N^{\frac{3}{2}}(1-N\thresh))^2} \)\\
	&\leq C\(\frac{ \frac{1}{ h^{11}\thresh^{\frac{7}{2}} }}{h^2\thresh^{\frac{3}{2}}}+\frac{ \frac{1}{h^3}\cdot\frac{1}{h^7\thresh^{\frac{3}{2}}} \cdot\frac{1}{\thresh^{\frac{1}{2}(1-\alpha)} }\cdot\frac{1}{h^{4\alpha}\thresh^{\frac{3\alpha}{2}}}} {h^4\thresh^3}\)=C\(\frac{1}{h^{13}\thresh^{5}}+\frac{1}{h^{14+4\alpha}\thresh^{5+4\alpha}}\)\leq \frac{C}{h^{18}\thresh^{9}}.
	\end{align*}	
	Finally, 
\begin{align*}
 \norm{D\pi}_{C^{0, \alpha} (\overline{\mathcal{K}^\thresh};\R^N)}\leq C(1+\norm{Dr}_{L^\infty(\mathcal{K}^\thresh)}+[Dr]_{C^{0, \alpha} (\overline{\mathcal{K}^\thresh})})\leq \frac{C}{h^{18}\thresh^{9}}
\end{align*}
by the calculation above combined with \eqref{eqn: Dr bound}
\end{proof}

With the above estimate, we can now prove linear convergence and locally superlinear convergence of our algorithm. This is done essentially as in \cite{KitagawaMerigotThibert19}.
\begin{proof}[Proof of Theorem \ref{thm: linear convergence}]
	Let $\psialg:=\psi_k$ be the vector chosen at the $k$th step of  Algorithm \ref{alg: damped newton}, $\diralg:=(Dw_{h, \thresh}(\psialg))^{-1}(w_{h, \thresh}(\psialg)-w)$, and define the curve $\psialg(t):=\pi(\psialg-t\diralg)$ (where $\pi$ is defined in Proposition \ref{prop: alg error bounds}). 
	We also take $\tilde{L}:= \norm{D\pi}_{C^{0, \alpha} (\overline{\mathcal{K}^\thresh};\R^N)}$, which has the bound claimed in the statement of the theorem by Proposition \ref{prop: alg error bounds}. As noted above $\psialg\in \mathcal{K}^\thresh\cap \mathcal{W}^{\epsilon_0}$, hence by Proposition \ref{prop: Dw estimates} we have the estimates \eqref{eqn: lipschitz bound w} and \eqref{eqn: inverse Dw bound}. Let $\tau_1:=\inf\{t\geq 0\mid \psialg(t)\not\in \mathcal{W}^{\frac{\epsilon_0}{2}}\}$, then $w_{h, \thresh}^j(\psialg(\tau_1))=\frac{\epsilon_0}{2}$ for some $1\leq j\leq N$, thus (using that $\psialg\in \normset{w, h, \thresh}$ so $\pi(\psialg)=\psialg$ and $\norm{\diralg}\leq \frac{\norm{w_{h, \thresh}(\psialg)-w}}{\kappa}$) we calculate
	\begin{align*}
		\frac{\epsilon_0}{2}&\leq \norm{w_{h, \thresh}(\psialg(\tau_1))-w_{h, \thresh}(\psialg)}\leq L\norm{\psialg(\tau_1)-\psialg}\\
		&=L\norm{\pi(\psialg-\tau_1\diralg)-\pi(\psialg)}\leq L\tilde{L}\tau_1\norm{\diralg}\leq \frac{L\tilde{L}\tau_1\norm{w_{h, \thresh}(\psialg)-w}}{\kappa}.
	\end{align*}
	The above gives a lower bound of $\frac{\kappa \epsilon_0}{2L\tilde{L}\norm{w(\psialg)-w}}$ on the first exit time $\tau_1$, and $w$ is uniformly $C^{1, \alpha}$ on the image $\psialg([0, \tau_1])$ while $\pi$ remains uniformly $C^{1, \alpha}$ on the segment $[\psialg, \psialg-\tau_1\diralg]$. We will now Taylor expand in $t$. Note that
	\begin{align*}
		\left. \frac{d}{dt}\right\vert_{t=0} &w_{h, \thresh}(\psialg(t))=\left.-Dw_{h, \thresh}(\psialg(t))\diralg+\inner{Dr(\psialg(t))}{\diralg}Dw_{h, \thresh}(\psialg(t))\onevect\right\vert_{t=0}\\
		&=-(w_{h, \thresh}(\psialg)-w)+\inner{Dr(\psialg)}{\diralg}Dw_{h, \thresh}(\psialg)\onevect.
	\end{align*}
	Using \eqref{eqn: grad r} and that $\psialg\in\normset{w, h, \thresh}$, we obtain
	\begin{align*}
		\inner{Dr(\psialg)}{\diralg}&=\frac{\inner{Dw_{h, \thresh}(\psialg)^T\onevect}{Dw_{h, \thresh}(\psialg)^{-1}(w_{h, \thresh}(\psialg)-w)}}{\inner{Dw_{h, \thresh}(\psialg)\onevect}{\onevect}}\\
		&=\frac{\inner{\onevect}{w_{h, \thresh}(\psialg)-w}}{\inner{Dw_{h, \thresh}(\psialg)\onevect}{\onevect}}=0.
	\end{align*}
	Now Taylor expanding we obtain
	\begin{align}\label{eqn: taylor expansion}
		w_{h, \thresh}(\psialg(t)) &= w_{h, \thresh}(\psialg(0)) + \bigg( \left. \frac{d}{du}\right\vert_{u=0} w_{h, \thresh}(\psialg(u)) \bigg)t + \int_{0}^t \(\left. \frac{d}{du}\right\vert_{u=s} w_{h, \thresh}(\psialg(u)) -  \left. \frac{d}{du}\right\vert_{u=0} w_{h, \thresh}(\psialg(u)) \) ds	\notag \\	
		&=:(1-t)w_{h, \thresh}(\psialg)+tw+R(t).
	\end{align}
	We see that
	\begin{align*}
	R^i(t) &= \int_0^t\( \inner{\nabla w^i_{h, \thresh}(\psialg(s))}{\dot{\psialg}(s)} - \inner{\nabla w^i_{h, \thresh}(\psialg(0))}{\dot{\psialg}(0)} \)ds \\
	&=\int_0^t \(\inner{\nabla w^i_{h, \thresh}(\psialg(s)) - \nabla w^i_{h, \thresh}(\psialg(0))}{\dot{\psialg}(s)} + \inner{\nabla w^i_{h, \thresh}(\psialg(0))}{\dot{\psialg}(s)- \dot{\psialg}(0)}\) ds.
	\end{align*}
	We will examine the two inner products separately. For $t\in [0, \tau_1]$ we have
	\begin{align*}
	&\int_0^t \inner{\nabla w^i_{h, \thresh}(\psialg(s)) - \nabla w^i_{h, \thresh}(\psialg(0))}{\dot{\psialg}(s)} ds 
	\leq \int_0^t \norm{\nabla w^i_{h, \thresh}(\psialg(s)) - \nabla w^i_{h, \thresh}(\psialg(0))} \norm{\dot{\psialg}(s)} ds \\
	&\leq \int_0^t ([Dw_{h, \thresh}]_{C^{0, \alpha}(\normset{w, h, \thresh})} \norm{\psialg(s) - \psialg(0)}^\alpha) (\norm{D\pi(\psialg-s\diralg)}\norm{\diralg}) ds \\
	&\leq \int_0^t ([Dw_{h, \thresh}]_{C^{0, \alpha}(\normset{w, h, \thresh})} \norm{D\pi}_{C^{0, \alpha} (\overline{\mathcal{K}^\thresh}) }^\alpha \norm{s\diralg}^{\alpha^2}) (\norm{D\pi(\psialg-s\diralg)}\norm{\diralg}) ds \\	
	&\leq \frac{L \ti L^{1+\alpha} \norm{\diralg}^{\alpha^2 + 1}}{\alpha^2 + 1 } t^{\alpha^2 + 1}
	\end{align*}
	and 
	\begin{align*}
	\int_0^t \inner{\nabla w^i_{h, \thresh}(\psialg(0))}{\dot{\psialg}(s)- \dot{\psialg}(0)} ds
	&\leq \int_0^t \norm{\nabla w^i_{h, \thresh}(\psialg(0))}\norm{\dot{\psialg}(s)- \dot{\psialg}(0)} ds \\
	&\leq \int_0^t \norm{Dw_{h, \thresh}(\psialg(0))}\norm{(D\pi(\psialg(0)) - D\pi(\psialg(s)))\diralg} ds \\	
	%&\leq \int_0^t \norm{Dw_{h, \thresh}(\psialg(0))}\norm{D\pi}_{C^{0, \alpha} (\overline{\mathcal{K}^\thresh})} \norm{\psialg(s) - \psialg(0)}^\alpha \norm{\diralg} ds \\	
		&\leq \int_0^t \norm{Dw_{h, \thresh}(\psialg(0))}\norm{D\pi}_{C^{0, \alpha}  (\overline{\mathcal{K}^\thresh}) }^{1+\alpha} \norm{s\diralg}^{\alpha^2} \norm{\diralg} ds \\	
	&= \frac{L \ti L^{1+\alpha} \norm{\diralg}^{1+\alpha^2}}{\alpha^2 + 1} t^{\alpha^2 + 1}
	\end{align*}
	where we have used $\dot{\psialg}(s) = -(D\pi(\psialg -s \diralg)) (\diralg)$. 
	Hence for $t\in [0, \tau_1]$ we obtain the bound on the remainder term $R$ above as
	\begin{align*}
		\norm{R(t)}
		\leq \frac{2L \ti L^{1+\alpha}\sqrt{N} \norm{\diralg}^{1+\alpha^2}}{\alpha^2 + 1} t^{\alpha^2 + 1} 
		\leq \frac{2L \ti L^{1+\alpha}\sqrt{N} \norm{w_{h, \thresh}(\psialg)-w}^{1+\alpha^2}}{\kappa^{1+\alpha^2}} t^{\alpha^2 + 1}.		
	\end{align*}
	At this point, the remainder of the proof proceeds exactly as that of \cite[Proposition 6.1]{KitagawaMerigotThibert19} following equation (6.3) there, with $w_{h, \thresh}$ replacing the map $G$ and $\alpha^2$ instead of $\alpha$. For the convenience of the reader we give the analogous expressions for $\tau_i$ which are
	\begin{align*}
	\tau_1 &\geq \frac{\kappa \eps_0}{2L\ti L \norm{w_{h, \thresh}(\psialg)-w}},\\
	\tau_2 &= \min(\tau_1, \frac{\kappa^{1+ \frac{1}{\alpha^2}} \eps_0^{\frac{1}{\alpha^2}}}{ (2 L\ti L^{1+\alpha} \sqrt{N})^{\frac{1}{\alpha^2}}   \norm{w_{h, \thresh}(\psialg)-w}^{1 + \frac{1}{\alpha^2}}   }),\\
	\tau_3 &= \min(\tau_2, \frac{\kappa^{1+ \frac{1}{\alpha^2}}}{(4L\ti L^{1+\alpha} \sqrt{N})^{\frac{1}{\alpha^2}} \norm{w_{h, \thresh}(\psialg)-w}   }  , 1).
	\end{align*}
	Finally, note that since $\sum_{i=1}^N w_{h, \thresh}(\psialg)^i = \sum_{i=1}^N w^i$, we have the bound $$\norm{w_{h, \thresh}(\psialg)-w}\leq 2\sum_{i=1}^N w^i\leq 2N.$$
	With these expressions, we can calculate
	\begin{align*}
	\conj \tau_k \leq 
	\frac{\eps_0^{\frac{1}{\alpha^2}} \kappa^{1+ \frac{1}{\alpha^2}}}{(4L\ti L^{1+\alpha} \sqrt{N})^{\frac{1}{\alpha^2}}  \norm{w_{h, \thresh}(\psialg_k)-w}^{1+\frac{1}{\alpha^2}}   }\leq \tau_3,
	\end{align*}
then global linear and local superlinear convergence follows as in \cite[Proposition 6.1]{KitagawaMerigotThibert19}.
\end{proof}

We conclude by using the above estimate Proposition \ref{prop: alg error bounds} to give a crude estimate on the number of iterations necessary to obtain an approximation of a solution to within an error of $\zeta$. Note that Corollary \ref{cor: number of iterations} is far from tight, as it does not take into account that our rate derived in Proposition \ref{prop: alg error bounds} goes to zero or that we have locally $1+\alpha^2$-superlinear convergence, but still serves as a starting point. 
\begin{cor}\label{cor: number of iterations}
	There exists a universal constant $C > 0$ so that for every $\zeta > 0$, and $\eps_0$, $h$, $\thresh$ sufficiently small depending on universal quantities, Algorithm \ref{alg: damped newton}, terminates in at most $\frac{\log \frac{\zeta}{2N}}{\log(1- \eta)}$ steps where $\eta =C{\eps_0^{1+ \frac{2}{\alpha^2}} h^{6 + \frac{18}{\alpha} + \frac{27}{\alpha^2}} \thresh^{\frac{3}{2} + \frac{9}{\alpha} + \frac{25}{2\alpha^2}}}$. 
\end{cor}

\begin{proof}
If $\conj\tau_k\neq 1$, we have
	\begin{align*}
	\conj \tau_k 
	&= 		\frac{\eps_0^{\frac{1}{\alpha^2}} \kappa^{1+ \frac{1}{\alpha^2}}}{(4L\ti L^{1+\alpha} \sqrt{N})^{\frac{1}{\alpha^2}}  \norm{w_{h, \thresh}(\psialg_k)-w}^{1+\frac{1}{\alpha^2}}   }
	\geq C	\frac{\eps_0^{\frac{1}{\alpha^2}} (\epsilon_0 h^{6}\thresh^{\frac{3}{2}})^{1+ \frac{1}{\alpha^2}}}{({(h^{-18}\thresh^{-9})^{1+\alpha}}\max(h^{-2}\thresh^{-2}, h^{-3}\thresh^{-\frac{1}{2}}))^{\frac{1}{\alpha^2}}}  \\
	&\geq C	\frac{\eps_0^{\frac{1}{\alpha^2}} (\epsilon_0 h^{6}\thresh^{\frac{3}{2}})^{1+ \frac{1}{\alpha^2}}}{({(h^{-18}\thresh^{-9})^{1+\alpha}} (h^{-3}\thresh^{-2}))^{\frac{1}{\alpha^2}}}  
	= C {\eps_0^{1+ \frac{2}{\alpha^2}} h^{6 + \frac{18}{\alpha} + \frac{27}{\alpha^2}} \thresh^{\frac{3}{2} + \frac{9}{\alpha} + \frac{25}{2\alpha^2}}},
	\end{align*}
and we may assume $h$, $\eps_0$, $\thresh$ are sufficiently small so that $1-\frac{C {\eps_0^{1+ \frac{2}{\alpha^2}} h^{6 + \frac{18}{\alpha} + \frac{27}{\alpha^2}} \thresh^{\frac{3}{2} + \frac{9}{\alpha} + \frac{25}{2\alpha^2}}}}{2}\geq \frac{1}{2}$. 
	Hence regardless of which value $\conj\tau_k$ takes at each iteration, after $\ell$ iterations we have
	\begin{align*}
	\norm{w(\psi_\ell) - w} \leq (1- \eta)^\ell \norm{w(\psi_0) - w} \leq 2N (1- \eta)^\ell
	\end{align*}
	where $\eta =  C{\eps_0^{1+ \frac{2}{\alpha^2}} h^{6 + \frac{18}{\alpha} + \frac{27}{\alpha^2}} \thresh^{\frac{3}{2} + \frac{9}{\alpha} + \frac{25}{2\alpha^2}}} $. Solving $(1- \eta)^\ell \norm{w(\psi_0) - w} \leq 2N (1- \eta)^\ell \leq \zeta$ for $\ell$, we see that it suffices to take
%	\begin{align*}
	$\ell \geq \frac{\log \frac{\zeta}{2N}}{\log(1- \eta)}$.
%	\end{align*}
\end{proof}

\section{Stability of Laguerre Cells} \label{section: Stability of Laguerre Cells}

In this section we prove that the convergence in our algorithm can be seen in terms of the Laguerre cells themselves instead of just in terms of the $w(\psi_k)$. 

\subsection{Proof of Theorem \ref{thm: symmetric convergence}}
We first prove $\mu$-symmetric convergence of Laguerre cells.

\begin{proof}[Proof of Theorem \ref{thm: symmetric convergence}]
	Let ${w}\in \R^N$ with $\sum_{i=1}^N {w}^i\geq 1$, ${w}^i\geq 0$ and $\psi_{h, \thresh}\in \mathcal{K}^\thresh$, and let $(T, \weightvect)$ be a pair minimizing \eqref{eqn: monge ver} with the storage fee function $F_{w}$. 
	Then if we define $\weightvect_{h, \thresh} := G(\psi_{h, \thresh})$ and $\conj{w}:=w_{h, \thresh}(\psi_{h, \thresh})$, by Proposition \ref{prop: constructing solutions}, the pair $(T_{\psi_{h, \thresh}}, \weightvect_{h, \thresh})$ minimizes \eqref{eqn: monge ver} with storage fee equal to $F_{\conj w, h, \thresh}$. By \cite[Theorem 4.7]{BansilKitagawa19a}, there also exists a pair $(T_{\conj{w}, \thresh}, \weightvect_{\conj{w}, \thresh})$ which minimizes \eqref{eqn: monge ver} with storage fee $F_{\conj{w}, 0, \thresh}$.
	Let
	\begin{align*}
	\mathcal{C}(\ti\weightvect) = \min_{S_\#\mu =\nu_{\ti \weightvect}} \int c(x, S(x)) d\mu = \sup_{\psi\in \R^N} \(-\int \psi^{c^*} d\mu - \inner{\psi}{\ti \weightvect}\).
	\end{align*}
Since
%	\begin{align*}
	$\mathcal{C}( \weightvect_{h, \thresh}) + F_{\conj{w}, h, \thresh}( \weightvect_{h, \thresh}) = \min_{\ti \weightvect\in\weightvectset} \(\mathcal{C}(\ti{\weightvect}) + F_{\conj{w}, h, \thresh}(\ti\weightvect)\) \leq \mathcal{C}( {\weightvect_{\conj{w}, \thresh}}) + F_{\conj{w}, h, \thresh}( {\weightvect_{\conj{w}, \thresh}})$,
%	\end{align*}
we have 
	\begin{align*}\mathcal{C}(\weightvect_{h, \thresh}) - \mathcal{C}( {\weightvect}_{\conj{w}, \thresh}) \leq F_{\conj{w}, h, \thresh}( {\weightvect}_{\conj{w}, \thresh}) - F_{\conj{w}, h, \thresh}(\weightvect_{h, \thresh}) \leq -F_{\conj{w}, h, \thresh}( {\weightvect}_{h, \thresh}) \leq h. \end{align*}
	Next by Corollary \ref{unicon} from the appendix, we have $\frac{1}{32C_LN} \norm{\weightvect_{h, \thresh} - {\weightvect}_{\conj{w},  \thresh}}^2 \leq \mathcal{C}(\weightvect_{h, \thresh}) - \mathcal{C}({\weightvect}_{\conj{w}, \thresh}) \leq h$ as ${\weightvect}_{\conj{w}, \thresh}$ is the minimizer of $\mathcal{C}$ on the convex set $\prod_{i=1}^N [\thresh, \conj{w}^i + \thresh]$, which can be seen from $F_{\conj{w}, 0, \thresh}=\ind(\cdot \ \vert \prod_{i=1}^N [\thresh, \conj{w}^i + \thresh])$. 
	
	Since the $l^1$ and $l^2$ norms on $\R^N$ are comparable,
	\begin{align*}
	\norm{\weightvect_{h, \thresh} - {\weightvect}_{\conj{w}, \thresh}}_1 \leq \sqrt{N} \norm{\weightvect_{h, \thresh} - {\weightvect}_{\conj{w}, \thresh}} \leq 4N\sqrt{2C_Lh}.
	\end{align*}
	Since $\sum_i w^i=1$, we see $(T, \weightvect)$ minimizes \eqref{eqn: monge ver} with storage fee $\ind(\cdot \ \vert \prod_{i=1}^N [0, {w}^i])$, hence by \cite[Theorem 2.6]{BansilKitagawa20a}, we obtain $\norm{ {\weightvect}_{\conj{w}, \thresh}-\weightvect }_1 \leq 2N\thresh + 2 \norm{\conj w -  w}_1$. By he triangle inequality,
	\begin{align*}
	\norm{G(\psi_{h, \thresh})-\weightvect}_1=\norm{\weightvect_{h, \thresh} - \weightvect}_1 \leq 2(N\thresh + \norm{\conj w - w}_1 + 2N\sqrt{2C_Lh}),
	\end{align*}
	proving \eqref{eqn: G bound}, and then \cite[Corollary 2.7]{BansilKitagawa20a} gives
	\begin{align*}
	\sum_{i=1}^N \Delta_\mu({\Lag_{i}(\psi_{h, \thresh})}, T^{-1}(\{y_i\})) \leq 8N(N\thresh +  \norm{\conj w - w}_1 + 2N\sqrt{2C_Lh})
	\end{align*}
	proving \eqref{eqn: mu diff bound}.
\end{proof}

\subsection{Proof of Theorem \ref{thm: hausdorff convergence}}
Next we prove convergence in terms of Hausdorff distance.

\begin{proof}[Proof of Theorem \ref{thm: hausdorff convergence}]
	We begin with statement (1). By \cite[Proposition 3.5, Proposition 4.4, Corollary 4.5]{BansilKitagawa19a}, there exists some $\psi\in \R^N$ such that $T=T_\psi$ $\mu$-a.e. and $\weightvect=G(\psi)$. Under the hypotheses of (1), by Theorem \ref{thm: symmetric convergence} \eqref{eqn: G bound}, we see that $\norm{G(\psi_k)-\weightvect}\to 0$ as $k\to\infty$. Then since minimizers of \eqref{eqn: monge ver} are minimizers of a classical optimal transport problem once the weight $\weightvect$ is known, we can apply \cite[Theorem 1.10]{BansilKitagawa20a} which gives the claim in (1). Claim (2) also follows immediately from \cite[Theorem 1.10]{BansilKitagawa20a}, since we know $\norm{G(\psi_{h, \thresh})-\weightvect}_1 \leq 2(N\thresh + \norm{\conj w - w}_1 + 2N\sqrt{2C_Lh})$ by Theorem \ref{thm: symmetric convergence} \eqref{eqn: G bound}.
\end{proof}

\section{Numerical examples}\label{section: numerical examples}
In this section we present some numerical examples produced by an implementation of  Algorithm \ref{alg: damped newton}. In each example, the source measure $\mu$ is supported on the 2D square $[0, 3]^2$, and the finite set $Y$ is a $30\times 30$ uniform grid of points with a random perturbation added, contained in the square $[0, 1]^2$. Each example was calculated to an error of $10^{-10}$, with parameters $h=\frac{1}{2}$ and $\epsilon=10^{-6}$; each figure below shows the boundaries of the associated Laguerre cells after various numbers of iterations. The code is based on a modification of the PyMongeAmpere interface developed by Quentin M\'erigot\footnote{M\'erigot's original code available at \url{https://github.com/mrgt/PyMongeAmpere}}. 

\begin{ex}\label{ex: KMT comparison}
In this example, the source measure $\mu$ has density identically zero on the square $[1, 2]^2$, identically equals a positive constant on the boundary of $[0, 3]^2$, and is linearly interpolated over a triangulation of $[0, 3]^2$ using $18$ triangles (see figure in \cite[Section 6.3]{KitagawaMerigotThibert19} for the triangulation, this measure is the same as what appears in that section), then normalized to unit mass. By a small modification of \cite[Appendix A]{KitagawaMerigotThibert19}, this $\mu$ satisfies a Poincar{\'e}-Wirtinger inequality. The weights $w$ are randomly generated, and taken to sum to one, so this example is a classical optimal transport problem. 

Algorithm \ref{alg: damped newton} reaches the specified error in $74$ iterations, while the algorithm of \cite{KitagawaMerigotThibert19} takes $62$ iterations, hence the two have comparable performance for classical optimal transport with source satisfying a Poincar{\'e}-Wirtinger inequality. The final diagram of Laguerre cells for both algorithms is presented in Figure \ref{fig: KMT comp} below (seeded with the same random values).
 \begin{figure}[H]
\begin{mdframed}
\centering
\begin{subfigure}{.33 \textwidth}
    \centering
    \frame{
    \includegraphics[width=0.7\textwidth]{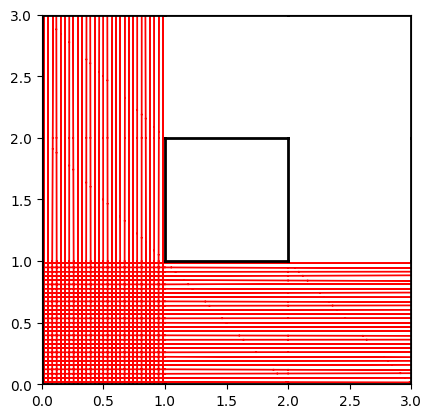}}
    \caption[short]{Algorithm \ref{alg: damped newton}: Iter$=0$}
\end{subfigure}%
\begin{subfigure}{.33\textwidth}
    \centering
    \frame{
    \includegraphics[width=0.7\textwidth]{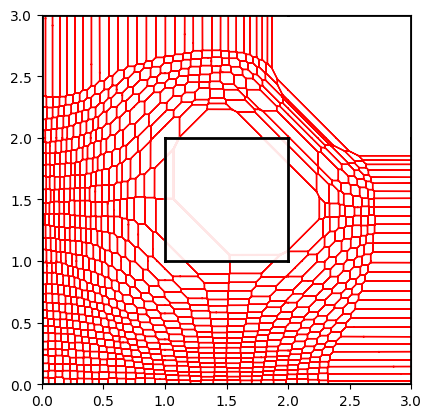}}
    \caption[short]{Algorithm \ref{alg: damped newton}: Iter$=50$}
\end{subfigure}%
\begin{subfigure}{.33\textwidth}
    \centering
    \frame{
    \includegraphics[width=0.7\textwidth]{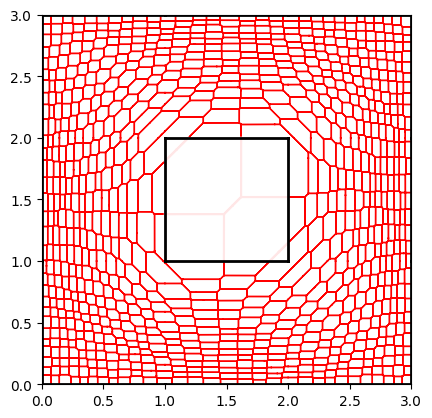}}
    \caption[short]{Algorithm \ref{alg: damped newton}: Iter$=74$}
\end{subfigure}
\raggedleft
\begin{subfigure}{.33 \textwidth}
    \centering
    \frame{
    \includegraphics[width=0.7\textwidth]{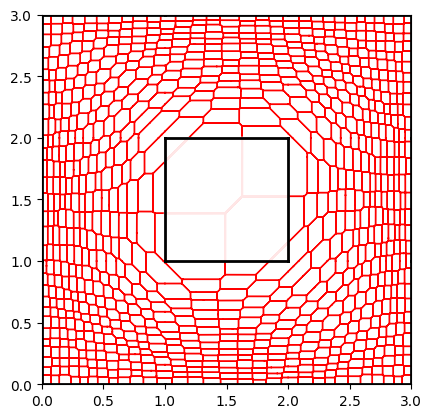}}
    \caption[short]{\cite{KitagawaMerigotThibert19}: Iter$=62$}
\end{subfigure}
\end{mdframed}
\caption{Laguerre cells of Example \ref{ex: KMT comparison}}\label{fig: KMT comp}
\end{figure}
\end{ex}

\begin{ex}\label{ex: random storage fee}
 In this example the source measure $\mu$ is the same as Example \ref{ex: KMT comparison}, and the vector $w$ associated to the storage fee is randomly generated. Since $\sum_{i=1}^Nw^i>1$ this is \emph{not} a classical optimal transport problem, but is an optimal transport problem with storage fee.
 
 Algorithm \ref{alg: damped newton} reaches the specified error tolerance in $57$ iterations. Attempting to run the algorithm from \cite{KitagawaMerigotThibert19} with a target measure given by the weights $w$ fails to reduce the error beyond $2\cdot 10^{-2}$ and produces dual vectors leading to clearly incorrect Laguerre cells. This is to be expected, as this example is not a classical optimal transport problem.
   \begin{figure}[H]
\begin{mdframed}
\centering
\begin{subfigure}{.33 \textwidth}
    \centering
    \frame{
    \includegraphics[width=0.7\textwidth]{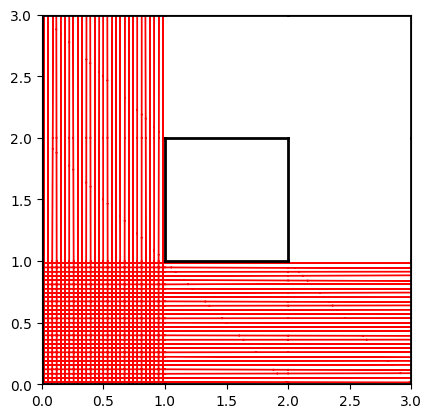}}
    \caption[short]{Algorithm \ref{alg: damped newton}: Iter$=0$}
\end{subfigure}%
\begin{subfigure}{.33\textwidth}
    \centering
    \frame{
    \includegraphics[width=0.7\textwidth]{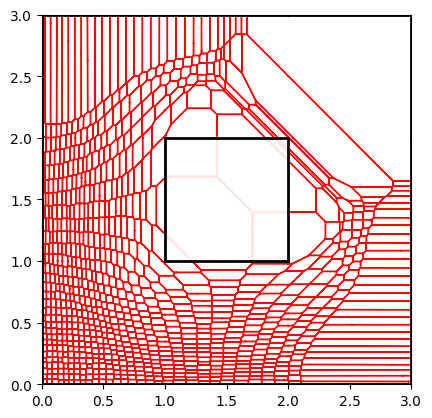}}
    \caption[short]{Algorithm \ref{alg: damped newton}: Iter$=25$}
\end{subfigure}%
\begin{subfigure}{.33\textwidth}
    \centering
    \frame{
    \includegraphics[width=0.7\textwidth]{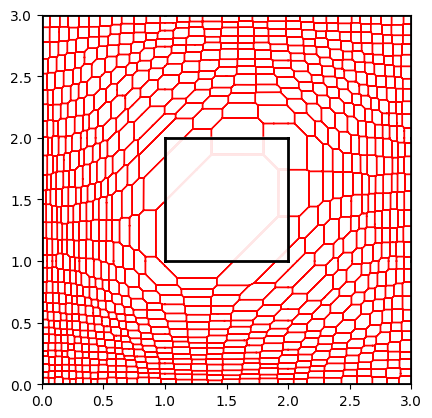}}
    \caption[short]{Algorithm \ref{alg: damped newton}: Iter$=57$}
\end{subfigure}
\end{mdframed}
\caption{Laguerre cells of Example \ref{ex: random storage fee}}\label{fig: storage}
\end{figure}

\end{ex}

\begin{ex}\label{ex: non-PW}
In this final example, the source measure $\mu$ is taken to have density identically zero on the strip $[1, 2]\times [0, 3]$, equal to a positive constant on the edges $\{0, 1\}\times [0, 3]$, and then is linearly interpolated over the same triangulation as in Example \ref{ex: KMT comparison} (and again normalized to unit mass). In particular, as $\spt\mu$ is not connected, this measure does \emph{not} satisfy a $(q, 1)$-Poincar{\'e}-Wirtinger inequality for any $q\geq 1$. The weights $w$ are taken with random weights summing to one, hence this corresponds to a classical optimal transport problem.

Algorithm \ref{alg: damped newton} reaches the error tolerance in $123$ iterations, while the algorithm from \cite{KitagawaMerigotThibert19} fails to produce any reduction of error from the initial state. This is due to the lack of a Poincar{\'e}-Wirtinger inequality for $\mu$.
   \begin{figure}[H]
\begin{mdframed}
\centering
\begin{subfigure}{.33 \textwidth}
    \centering
    \frame{
    \includegraphics[width=0.7\textwidth]{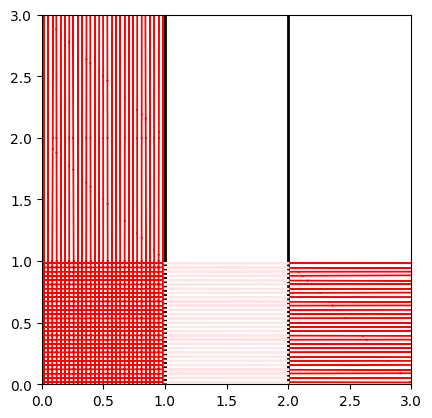}}
    \caption[short]{Algorithm \ref{alg: damped newton}: Iter$=0$}
\end{subfigure}%
\begin{subfigure}{.33\textwidth}
    \centering
    \frame{
    \includegraphics[width=0.7\textwidth]{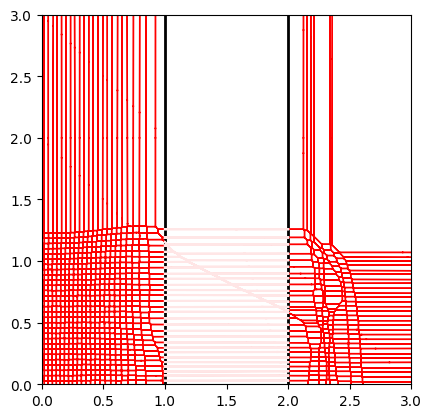}}
    \caption[short]{Algorithm \ref{alg: damped newton}: Iter$=50$}
\end{subfigure}%
\begin{subfigure}{.33\textwidth}
    \centering
    \frame{
    \includegraphics[width=0.7\textwidth]{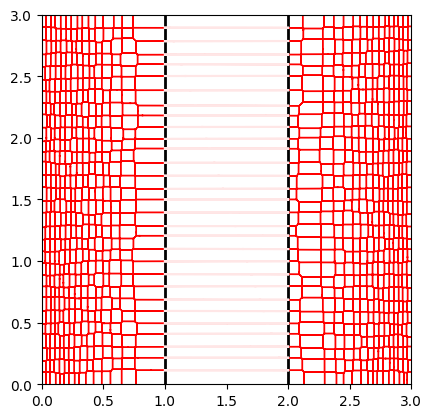}}
    \caption[short]{Algorithm \ref{alg: damped newton}: Iter$=123$}
\end{subfigure}
\end{mdframed}
\caption{Laguerre cells of Example \ref{ex: non-PW}}\label{fig: non-PW}
\end{figure}
\end{ex}
\begin{appendices}
	\section {Strong Convexity of $\mathcal{C}$}\label{appendix: strong convexity}

	\begin{lem}\label{lem: strong convex C}
		
		$\mathcal{C}$ is strongly convex. In particular
		\begin{align*}
		t\mathcal{C}(x) + (1-t)\mathcal{C}(y)  \geq  \mathcal{C}(tx + (1-t)y) + \frac{1}{8C_LN}t(1-t) \norm{y-x}^2,
		\end{align*}
		where $[G]_{C^{0, 1}(\R^N)}\leq C_LN$, and $C_L>0$ is universal.
	\end{lem}
	
	\begin{proof}
		
		Let
%		\begin{align*}
		$B(\psi) = \int \psi^{c^*} d\mu$.
%		\end{align*}
		We see that $\mathcal{C}(\weightvect) = B^*(-\weightvect)$; also by \cite{AbedinGutierrez17} $B$ is $C^{1,1}$, $\nabla B=-G$, and $B$ is convex (see \cite[Theorem 1.1]{KitagawaMerigotThibert19}). By \cite[Theorem 5.1]{AbedinGutierrez17} we see the Lipschitz constant of $G$ is bounded from above by $C_LN$ where $C_L>0$ is some universal constant. Now
		\begin{align*}
		0 &\leq tB(x) + (1-t)B(y) - B(tx + (1-t)y) \\
		&= tB(x) + (1-t)\bigg(B(x) + \inner{y-x}{\nabla B(x)} + \int_0^1 \inner{\nabla B((1-s)x+sy)-\nabla B(x)}{y-x} ds\bigg) \\
		&- \bigg(B(x) + \inner{tx+(1-t)y - x}{\nabla B(x)} \\
		&\quad+(1-t) \int_0^{1} \inner{\nabla B((1-s(1-t))x+s(1-t)y)-\nabla B(x)}{y-x} ds\bigg) \\
		&\leq(1-t)\int_0^1 \norm{\nabla B((1-s)x+sy)-\nabla B(x)}\norm{y-x} ds \\
		&+(1-t) \int_0^{1} \norm{\nabla B((1-s(1-t))x+s(1-t)y)-\nabla B(x)}\norm{y-x} ds\\
		&\leq C_LN(1-t)\(\int_0^1 s\norm{y-x}^2 ds+(1-t) \int_0^{1} s\norm{y-x}^2 ds\)\\
		&\leq (1-t) {C_LN} \norm{y-x}^2 .
		\end{align*}
		By repeating a similar argument we get $tB(x) + (1-t)B(y) - B(tx + (1-t)y) \leq t {C_LN} \norm{y-x}^2$. Hence $tB(x) + (1-t)B(y) - B(tx + (1-t)y) \leq 2C_LN t(1-t) \norm{y-x}^2$.
		
		In the terminology of \cite[Definition 1]{AzePenot95}, we have shown that $B$ is $\sigma$-smooth where $\sigma(x) := 2C_LNx^2$. 
		Since it is well-known that $\sigma^*(z) = \frac{1}{8C_LN}{z}^2$,
		by \cite[Proposition 2.6]{AzePenot95} we see that $\mathcal{C}$ is $\sigma^*$- convex, i.e.
		$
		t\mathcal{C}(x) + (1-t)\mathcal{C}(y)  \geq  \mathcal{C}(tx + (1-t)y) + \frac{1}{8C_LN}t(1-t) \norm{y-x}^2,
		$
		finishing the proof.	
	\end{proof}
	
	\begin{cor} \label{unicon}
		
		Let $K$ be a convex subset of the domain of $\mathcal{C}$. Let $\weightvect_{min}$ be the minimizer of $\mathcal{C}$ on $K$ and $\weightvect \in K$ be arbitrary. Then
		$
		\mathcal{C}(\weightvect) - \mathcal{C}(\weightvect_{min}) \geq \frac{1}{32C_LN} \norm{\weightvect-\weightvect_{min}}^2.
		$
	\end{cor}
	
	\begin{proof}
		
		By choice of $\weightvect_{min}$, we have $\frac 12 \mathcal{C}(\weightvect) \geq \frac 12 \mathcal{C}(\weightvect_{min}) $ and $- \mathcal{C}(\weightvect_{min}) \geq - \mathcal{C}(\frac{1}{2}(\weightvect + \weightvect_{min}))$. Hence by the above lemma we have
		$
		\mathcal{C}(\weightvect) - \mathcal{C}(\weightvect_{min}) \geq \frac{1}{2}(\mathcal{C}(\weightvect) + \mathcal{C}(\weightvect_{min})) - \mathcal{C}(\frac{1}{2}(\weightvect + \weightvect_{min})) \geq \frac{1}{32C_LN} \norm{\weightvect-\weightvect_{min}}^2.
		$
	\end{proof}

\end{appendices}

\bibliographystyle{alpha}
\bibliography{snowshovelingalg}

\end{document}